\documentclass[11pt]{amsart}
\usepackage{amsmath,amsmath,latexsym,amssymb,amsfonts,amsbsy, amsthm}
\usepackage{mathrsfs}
\usepackage{fancyhdr}
\usepackage{latexsym}
\usepackage{bbding}
\usepackage{wasysym}
\usepackage{cite}
\usepackage{multicol,graphics}
\usepackage{graphicx}
\usepackage{color}
\usepackage{bm}
\usepackage{enumitem}

\numberwithin{equation}{section}
\newtheorem{theorem}{Theorem}[section]
\newtheorem{lemma}{Lemma}[section]

\newtheorem{corollary}{Corollary}[section]

\newcommand{\beq}{\begin{eqnarray}}
\newcommand{\eeq}{\end{eqnarray}}
\newcommand{\beqno}{\begin{eqnarray*}}
\newcommand{\eeqno}{\end{eqnarray*}}

\newenvironment {Proof of Theorem} {\noindent {\bf Proof of Theorem 2.1}}{\quad $\square$\par\vspace{3mm}}

\allowdisplaybreaks

\begin{document}
\newcommand{\D}{\displaystyle}
%\begin{frontmatter}



\title[Decay rate of liquid crystal flows]{Time decay rate of global strong solutions to nematic liquid crystal flows in $\mathbb R^3_+$\emph{}}

% \author[label1,label2]{<author name>}
% \address[label1]{<address>}
% \address[label2]{<address>

\author{Jinrui Huang}
\author{Changyou Wang}
\author{Huanyao Wen}

%\cortext[cor3]{Corresponding author. Email: wang2482@purdue.edu}
\address{School of Mathematics and Computational Science, Wuyi University, Jiangmen 529020, China}
\address{Department of Mathematics, Purdue University, West Lafayette IN 47907, USA}
\address{School of Mathematics, South China University of Technology, Guangzhou 510641, China}

\begin{abstract}
In this paper, we obtain optimal
time-decay rates in $L^r(\mathbb R^3_+)$ for $r\ge 1$ of global strong solutions to the 
nematic liquid crystal flows in $\mathbb R^3_+$, provided the initial 
data has small $L^3(\mathbb R^3_+)$-norm.
\end{abstract}
%\\
%\noindent{\it Research highlights:} \ The free boundary condition of the compressible liquid crystal flows is derived. We use the Schuader fixed point theorem to prove the local existence of the unique classical solution. We give some a priori estimates to the unique global classical solutions in Lagrangian coordinate.

\keywords
{Global strong solution, time-decay rate,
nematic liquid crystal flow}
\subjclass [2010] {76A15, 76N10, 74H40,
35Q30}
\maketitle
%\end{frontmatter}
\tableofcontents

\setcounter{section}{0} \setcounter{equation}{0}
\section{Introduction and statement of main results}

In this paper, we study a simplified nematic liquid crystal flow in the upper half three space
$\mathbb{R}^3_+=\big\{x=(x_1,x_2,x_3)\in \mathbb{R}^3:x_3>0\big\}$:
\begin{eqnarray} \label{system}
\begin{cases}
     u_t+u\cdot \nabla u+\nabla p=\mu\Delta{u}-\lambda\nabla\cdot(\nabla d\odot\nabla d),\\
     \nabla\cdot u=0,\\
     d_t+u\cdot\nabla d=\theta(\Delta{d}+|\nabla d|^2d),%\\[2mm]
\end{cases}
\end{eqnarray} where $u:\mathbb{R}^3\mapsto\mathbb{R}^3$ denotes the fluid velocity field,
$d:\mathbb{R}^3\mapsto\mathbb{S}^2\equiv\{y\in\mathbb R^3: |y|=1\big\}$ denotes
the macroscopic orientation field of liquid
crystal molecules, $p$ denotes the pressure function,
$\nabla d\odot\nabla d=(\langle\nabla_i d,\nabla_j d\rangle)_{1\le i, j\le 3}$,
and $\mu,\lambda, \theta>0$ represent the fluid viscosity, the competition between kinetic energy and potential energy,
and the microscopic elastic relaxation time for the molecular orientation field repsectively.
 The system (\ref{system}) is equipped with the following initial and boundary conditions:
\begin{eqnarray} \label{i-b condition}
\begin{cases}
     u=\frac{\partial d}{\partial x_3}=0,\ {\rm on}\ \partial\mathbb R^3_+\times (0,\infty),\\
     (u, d)\rightarrow (0,e_3), \ {\rm{as}}\ |x|\rightarrow \infty,\\
     (u, d)\big|_{t=0} =(u_0, d_0),\ {\rm in}\ \mathbb{R}^3_+,
\end{cases}
\end{eqnarray}
where $e_3=(0,0,1)\in \mathbb{S}^2$.

The system (\ref{system}) couples the forced Navier-Stokes equation with the transported flow of harmonic maps
to $\mathbb S^2$, which has attracted considerable interests recently. The rigorous mathematical analysis
of (\ref{system}) was first made by Lin-Liu \cite{Lin-2, Lin-3}, in which they considered
the Ginzburg-Landau approximation of \eqref{system}
by replacing  $|\nabla d|^2d$ by $\frac{1}{\epsilon^2}(1-|d|^2)d$ ($\epsilon>0$) and proved
the existence of global weak solutions and their partial regularities.
For the original system (\ref{system}), Lin-Lin-Wang \cite{Lin-1}  have established the existence of
a global weak solution that is smooth away from at most finitely many time in dimension two (see also \cite{Hong-1},
Hong-Xin \cite{Hong-2}, Huang-Lin-Wang \cite{huang}, Li-Lei-Zhang\cite{Lei}, Wang-Wang \cite{Wang-Wang}
for relevant results in dimension two).
In dimension three, while the existence of global weak solutions of (\ref{system}) remains an open problem,
there has been some interesting progress.
For example,  Ding-Wen \cite{Wen-Ding} have obtained the local existence and uniqueness of strong solutions in dimension three, Huang-Wang \cite{huang-wang} have provided a blow-up criterion of strong solutions, and the
well-posedness of \eqref{system} for an initial data $(u_0,d_0)$ with small $BMO^{-1}\times BMO$-norm
and with small $L^3_{uloc}(\mathbb R^3)$-norm has been shown by Wang \cite{Wang}
and Hineman-Wang \cite{Hineman} respectively. Most recently, Lin-Wang \cite{Lin-Wang16}
have shown the existence of global weak solutions in dimension three under the assumption
that the initial director field $d_0(\Omega)\subset\mathbb S^2_+$.
Concerning the long time asymptotical behavior of global strong solutions to \eqref{system} in $\mathbb R^3$,
Liu-Xu \cite{Liu} have established an optimal decay rate for $\|(u,\nabla d)\|_{H^m(\mathbb R^3)}$ under the
assumption that $(u_0,d_0)\in H^m(\mathbb{R}^3)\times H^{m+1}(\mathbb{R}^3,\mathbb{S}^2)$ ($m\geq3$)
has sufficiently small $\|(u_0,\nabla d_0)\|_{L^2(\mathbb R^3)}$-norm; while
Dai, and her coauthors, has obtained in  \cite{Dai,Dai-2} optimal decay rates in $H^m(\mathbb{R}^3)$
provided $\|u_0\|_{H^1(\mathbb R^3)}+\|d-e_3\|_{H^2(\mathbb R^3)}$ is sufficiently small.

A natural question is to ask for the large time asymptotical behavior of global solutions of \eqref{system}
on general domains. As a first step, 
we consider in this paper time decay rates in $L^p(\mathbb R^3_+)$ of
strong solutions of (\ref{system})-(\ref{i-b condition}) on the upper half space $\mathbb R^3_+$. This consideration
is also partly motivated by previous works  on the corresponding Navier-Stokes equations on $\mathbb R^3_+$,
which has been relatively well understood. For example,
the long time behavior of weak and strong solutions of (\ref{system}) in $L^p(\mathbb{R}^n_+)$
has been investigated by
Bae-Choe \cite{Bae}, Borchers-Miyakawa \cite{Bocher}, Fujigaki-Miyakawa \cite{Fujigaki}, Kozono \cite{Kozono}
in $p\in (1,+\infty)$, and by Han \cite{Han1,Han2,Han3} for the end point case $p=1$, which imposes
difficulties due to the unboundedness of the Leray projection operator $\mathbb{P}: L^1(\mathbb R^n_+)\rightarrow L^1_\sigma(\mathbb R^n_+)$. For the nematic liquid crystal flow \eqref{system}, the  super-critical nonlinearity $\nabla\cdot(\nabla d\odot\nabla d)$ in the momentum equation \eqref{system}$_1$ introduces new difficulties in establishing time decay estimates for solutions to \eqref{system} in $\mathbb{R}^3_+$. In particular,
\begin{itemize}
\item While the scaling of $\nabla d$ is comparable with $u$, 
the required estimates on $\nabla\cdot(\nabla d\odot\nabla d)$ is more delicate than the convective term $u\cdot\nabla u$,
because $\nabla d$ is not divergence free. In fact, third order derivatives of $d$ emerge in the estimate of
$\|\mathbb{P}\left(u\cdot\nabla u+\nabla\cdot(\nabla d\odot\nabla d)\right)\|_{L^1(\mathbb R^3_+)}$, which is equivalent to
the estimate of $\|\nabla d(t)\|_{H^1(\mathbb R^3_+)}^2+\||\nabla d(t)||\nabla^3d(t)|\|_{L^1(\mathbb R^3_+)}$. 
Therefore, higher order estimates of global solutions $(u,d)$ are needed.
To achieve this, we utilize an iteration argument to derive the basic $L^2$-decay estimate by first 
establishing $\|\nabla d(t)\|_{L^2(\mathbb R^3_+)}\lesssim t^{-1}$ through a continuity argument, and then improving 
it to $t^{-\frac54+\epsilon}$ ($\epsilon>0$), and finally to $t^{-\frac54}$. 
\end{itemize}

We would also like to point out that
\begin{itemize}
\item[(i)] in contrast with \cite{Dai,Dai-2,Liu} where they considered \eqref{system} on $\mathbb R^3$,
here we consider \eqref{system} on $\mathbb R^3_+$ and hence we have to analyze  the boundary contributions 
of global solutions, and 
\item[(ii)] the time decay estimate in $L^p(\mathbb R^3_+)$  in this paper holds  for any initial data $(u_0, d_0)\in 
L^3_\sigma(\mathbb R^3_+)\times 
\dot{W}^{1,3}(\mathbb R^3_+,\mathbb S^2)$ that has small  $\|(u_0, \nabla d_0)\|_{L^3(\mathbb R^3_+)}$ norm,  
which improves the conditions on the initial data given by \cite{Dai,Dai-2,Liu}.
\end{itemize}

In order to state the main results, we first recall some notations.
Denote by $C^\infty_{0,\sigma}(\mathbb R^3_+, \mathbb R^3)$ the space of smooth divergence-free vector fields with compact supports in $\mathbb{R}^3_+$,
and $L^r_\sigma(\mathbb R^3_+,\mathbb R^3)$, $r\in [1,\infty)$, the $L^r$-closure of $C^\infty_{0,\sigma}(\mathbb R^3_+, \mathbb R^3)$ in $L^r(\mathbb R^3_+,\mathbb R^3)$.
For any nonnegative integer $k$ and $r\in [1,\infty)$, denote by $W^{k,r}(\mathbb{R}^3_+)$ the $(k,r)$-Sobolev space in $\mathbb{R}^3_+$,
and $W^{k,r}_0(\mathbb R^3_+)$ the $W^{k,r}$-closure of the set $C^\infty_0(\mathbb R^3_+)$,
and
$$W^{k,r}(\mathbb R^3_+, \mathbb S^2)
=\Big\{ v\in W^{k,r}(\mathbb R^3_+, \mathbb R^3):
\ v(x)\in \mathbb S^2\ {\rm{for\ a.e.}}\ x\in \mathbb R^3_+\Big\}.
$$
Set
$$D^{k,r}(\mathbb{R}^3_+)=\left\{v\in L^1_{\rm{loc}}(\mathbb{R}^3_+):
\ \|\nabla^kv\|_{L^r(\mathbb R^3_+)}<\infty\right\}, $$
and $D^k(\mathbb R^3_+)=D^{k,2}(\mathbb R^3_+)$.

Our first theorem concerns the existence of a unique global strong
solution of \eqref{system} and its time-decay rate. More precisely, we have

\begin{theorem} \label{th:E_0} There exists an $\varepsilon_0>0$
such that  if
$u_0\in L^r_{\sigma}(\mathbb R^3_+,\mathbb R^3)$ for $r=2, 3$ and
$d_0\in D^{1}(\mathbb R^3_+,\mathbb S^2)$ satisfies
$$\|u_0\|_{L^3(\mathbb R^3_+)}+\|\nabla d_0\|_{L^3(\mathbb R^3_+)}
\le\varepsilon_0,$$
then the system (\ref{system})-(\ref{i-b
condition}) admits a unique global strong solution
$(u,d)$ such that for
any  $\tau>0$,  the following properties hold:
\begin{eqnarray}
\begin{cases}
     u\in C([0,\infty), L^2(\mathbb R^3_+)\cap L^2(
     [0,\infty), D^1(\mathbb R^3_+)),\\
     u\in C(\mathbb R_+, D^2(\mathbb R^3_+))\cap
     L^2([\tau,+\infty), D^3(\mathbb R^3_+)),\\
     u_t\in C(\mathbb R_+, L^2(\mathbb R^3_+))\cap
     L^2([\tau,+\infty), W^{1,2}(\mathbb R^3_+)),\\
     u\in L^\infty([0,\infty), L^3(\mathbb R^3_+)),\, \nabla d\in L^\infty([0,\infty),L^3(\mathbb R^3_+)),\\
     \big(|u|^{\frac32},  |\nabla d|^{\frac32}\big)
     \in L^2([0,\infty), D^1(\mathbb R^3_+)),\\
     \nabla d\in C([0,\infty), L^2(\mathbb R^3_+))\cap
     L^2([0,\infty), D^1(\mathbb R^3_+)),\\
     \nabla d\in C(\mathbb R_+, D^2(\mathbb R^3_+))
     \cap L^2([\tau,+\infty), D^3(\mathbb R^3_+)), \\
     d_t\in L^2([0,\infty), L^2(\mathbb R^3_+)),\\
     d_t\in C(\mathbb R_+, W^{1,2}(\mathbb R^3_+))\cap L^2([\tau,+\infty), D^2(\mathbb R^3_+)).
\end{cases}
\end{eqnarray}
If, in addition,  $u_0,d_0-e_3\in L^1(\mathbb R^3_+)$, then we have the following
decay estimates:
\begin{eqnarray}
\begin{cases}
     \|u(\cdot,t)\|_{L^{r}(\mathbb R^3_+)}
     +\|(d-e_3)(\cdot,t)\|_{L^r(\mathbb R^3_+)}\leq Ct^{-\frac32\left(1-\frac{1}{r}\right)},\\
     \|\nabla d(\cdot,t)\|_{L^s(\mathbb R^3_+)}\leq Ct^{-\frac12-\frac32\left(1-\frac{1}{s}\right)},\\
     \|\nabla u(\cdot,t)\|_{L^p(\mathbb R^3_+)}\leq Ct^{-\frac12-\frac32\left(1-\frac{1}{p}\right)},\\
     \|\nabla^2d(\cdot,t)\|_{L^{q}(\mathbb R^3_+)}\leq Ct^{-1-\frac32\left(1-\frac{1}{q}\right)},
\end{cases}
\end{eqnarray} for any $t>0$, $r\in (1,\infty]$, $s\in [1,\infty]$, $p\in (1,6]$, and $q\in [1,6]$.
\end{theorem}
Now we state the second main result of this paper.

\begin{theorem} \label{th:E_1} Under the same assumptions of Theorem \ref{th:E_0},
if, in addition, $u_0\in D^1(\mathbb R^3_+,\mathbb R^3)$ and $\nabla d_0\in D^1(\mathbb R^3_+)$,
then 
\begin{eqnarray}
    \|\nabla u(\cdot,t)\|_{L^1(\mathbb R^3_+)}\leq Ct^{-\frac12}
\end{eqnarray}
holds for any $t>0$.
\end{theorem}

\begin{corollary} \label{th:E_2} Under the same assumptions of Theorem \ref{th:E_0}, if,  in addition, 
\begin{eqnarray} \label{weight}
    \int_{\mathbb R^3_+}x_3|u_0(x)|{\rm d}x<\infty,
\end{eqnarray} then the estimates on $u$ can be improved into
\begin{eqnarray}\label{u-estimate}
\begin{cases}
     \|u(\cdot,t)\|_{L^{r}(\mathbb R^3_+)}\leq Ct^{-\frac12-\frac32\left(1-\frac{1}{r}\right)},\\
     \|\nabla u(\cdot,t)\|_{L^p(\mathbb R^3_+)}\leq Ct^{-1-\frac32\left(1-\frac{1}{p}\right)},\\
\end{cases}
\end{eqnarray} hold for any $t>0$, $r\in (1,\infty]$, and $q\in (1,6]$.
\end{corollary}

It remains to be an interesting question whether the director field $d$ satisfies improved estimates
on $\nabla d$, similar to \eqref{u-estimate}, provided $\int_{\mathbb R^3_+} x_3|\nabla d_0(x)|{\rm d}x<\infty$. 
 
The strong solutions of \eqref{system} from Theorems \ref{th:E_0} and \ref{th:E_1}
obey Duhamel's formula:
\begin{eqnarray} \label{u}
\begin{cases}
   \displaystyle u(t)=e^{-t\mathbb{A}}u_0-\int_0^te^{-(t-s)\mathbb{A}}\mathbb{P}\left(u(s)\cdot\nabla u(s)
   +\nabla\cdot(\nabla d(s)\odot\nabla d(s)\right){\rm d}s,\\
    \displaystyle(d-w_0)(t)=e^{t\Delta}(d_0-w_0)-\int_0^te^{(t-s)\Delta}\left(u(s)\cdot\nabla d(s)-|\nabla d(s)|^2d(s)\right){\rm d}s,
\end{cases}
\end{eqnarray}
and
\begin{eqnarray}
\begin{cases} \label{u-halft}
   \displaystyle u(t)=e^{-\frac{t}{2}\mathbb{A}}u(\frac{t}2)-\int_{\frac{t}{2}}^te^{-(t-s)\mathbb{A}}\mathbb{P}\left(u(s)\cdot\nabla u(s)+\nabla\cdot(\nabla d(s)\odot\nabla d(s)\right){\rm d}s,\\
    \displaystyle (d-w_0)(t)=e^{\frac{t}{2}\Delta}(d-w_0)(\frac{t}2)-\int_{\frac{t}{2}}^te^{(t-s)\Delta}\left(u(s)\cdot\nabla d(s)-|\nabla d(s)|^2d(s)\right){\rm d}s,
\end{cases}
\end{eqnarray}
where $\mathbb A=-\mathbb P\Delta$ is the Stokes operator.

Since the values of $\mu, \lambda$, and $\theta$  do not play any role in this paper,
we will henceforth assume $\mu=\lambda=\theta=1$.

\section{Preliminary estimates}

In this section, we will provide a few basic estimates related to the Stokes operator
$\mathbb A$. We start with the $L^p-L^q$ estimate for the Stokes semigroup, which
can be found in \cite{Fujigaki}.
\begin{lemma}  \label{le:1} For $n\ge 2$ and $1\le q<\infty$,
let $a\in L^q_{\sigma}(\mathbb{R}^n_+, \mathbb R^n)$, then for any non-negative integer
$k$,
it holds
\begin{eqnarray} \label{p-q}
   \|\nabla^ke^{-t\mathbb{A}}a\|_{L^p(\mathbb{R}^n_+)}\leq C_{k,p,q,n}t^{-\frac{k}{2}-\frac{n}{2}(\frac{1}{q}-\frac{1}{p})}\|a\|_{L^q(\mathbb{R}^n_+)}, \ \forall\ t>0,
\end{eqnarray}
where $C_{k,p,q,n}>0$ is independent of $a$, provided either $1\leq q<p\leq\infty$ or $1<q\leq p<\infty$.

Furthermore,
\begin{eqnarray} \label{p-q-2}
   \|\nabla e^{-t\mathbb{A}}a\|_{L^1(\mathbb{R}^n_+)}\leq C_{1,n}t^{-\frac{1}{2}}\|a\|_{L^1(\mathbb{R}^n_+)}, \ \forall\ t>0,
\end{eqnarray}
and (\ref{p-q}) and (\ref{p-q-2}) still hold, if we replace the operator $e^{-t\mathbb{A}}$
by $e^{t\Delta}$ with $a\in L^q(\mathbb{R}^n_+, \mathbb{R}^n)$.

%We also remark that for $k=1$ and $p=2$, if $\nabla a\in L^2$, then the above estimate can be rewritten as follows:
%\begin{eqnarray} \label{p-q-2}
%   \|\nabla e^{-t\mathbb{A}}a\|_{L^2(\mathbb{R}^n_+)}\leq C_{q,n,\|a\|_{L^q(\mathbb{R}^n_+)}}(1+t)^{-\frac{k}{2}-\frac{n}{2}\left(\frac{1}{q}-\frac{1}{2}\right)}.
%\end{eqnarray} In fact, for $t<1$, by the energy law for dynamical Stokes equations, $\|\nabla e^{-t\mathbb{A}}a\|_{L^2}\leq \|\nabla a\|_{L^2}\leq C$, then we finished the proof. For $t\geq1$, $\|\nabla e^{-t\mathbb{A}}a\|_{L^2(\mathbb{R}^n_+)}\leq C_{q,n,\|a\|_{L^q(\mathbb{R}^n_+)}}(1+t)^{-\frac{k}{2}-\frac{n}{2}\left(\frac{1}{q}-\frac{1}{2}\right)}\cdot\left(\frac{t}{1+t}\right)^{-\frac{k}{2}-\frac{n}{2}\left(\frac{1}{q}-\frac{1}{2}\right)}$, then we get (\ref{p-q-2}).
\end{lemma}

%\medskip

Recall that the Stokes operator is defined by
$$\mathbb{A}=-\mathbb{P}\Delta:D(\mathbb{A})\mapsto
L^2_\sigma(\mathbb{R}^3_+, \mathbb R^3),$$
where $D(\mathbb{A})=H^2(\mathbb{R}^3_+,\mathbb R^3)\cap
H^1_{0,\sigma}(\mathbb{R}^3_+,\mathbb R^3)$, and
$$\mathbb{P}: L^r(\mathbb{R}^3_+,\mathbb R^3)\mapsto L^r_\sigma(\mathbb{R}^3_+,\mathbb R^3)$$
is the Leray projection operator,
which is bounded for any $1<r<\infty$. It is well known that $\mathbb{A}$ is a
positive, self-adjoint operator on
$D(\mathbb{A})\subseteq L^2_\sigma(\mathbb{R}^3_+,\mathbb R^3)$,
and there exists a uniquely determined resolution $\{E_\lambda:\lambda\geq
0\}$ of identity in $L^2_\sigma(\mathbb{R}^3_+,\mathbb R^3)$ such that
$\mathbb{A}$ admits a spectral representation
\begin{eqnarray} \label{A}
    \mathbb{A}=\int_0^\infty \lambda{\rm d}E_\lambda
\end{eqnarray} so that
\begin{eqnarray}
    \|\mathbb{A}u\|_{L^2}^2=\int_0^\infty\lambda^2{\rm d}\|E_\lambda u\|_{L^2}^2,
    \ {\rm for}\ u\in L^2_\sigma(\mathbb{R}^3_+,\mathbb R^3).
\end{eqnarray}
For $0\le \lambda_0\le \infty$, let
$E_{\lambda_0}=s-\lim\limits_{\lambda\rightarrow\lambda_0}E_\lambda$
be the strong limit of operators.
From \cite{Sohr}, $\{E_\lambda:\lambda\geq
0\}$ satisfies the following properties:
\begin{itemize}
\item [(i)] $E_\lambda E_\mu=E_\mu E_\lambda=E_\lambda$ for $0\leq\lambda\leq\mu<+\infty$;

\item [(ii)] $E_\lambda=s-\lim\limits_{\mu\rightarrow\lambda}E_\mu$ for $0<\lambda<\mu<+\infty$;

\item [(iii)] $E_0=0$, and $s-\lim\limits_{\mu\rightarrow\infty}E_\mu=I$, the identity operator.
\end{itemize}
For any $\alpha\in (0,1)$, define the fractional order of Stokes operator $\mathbb A^\alpha$ by
\begin{eqnarray}
    \mathbb{A}^\alpha=\int_0^\infty \lambda^\alpha{\rm d}E_\lambda
\end{eqnarray} with
\begin{eqnarray}
    \|\mathbb{A}^\alpha u\|_{L^2}^2=\int_0^\infty\lambda^{2\alpha}{\rm d}\|E_\lambda u\|_{L^2}^2,\ {\rm for}\ u\in L^2_\sigma(\mathbb{R}^3_+,\mathbb R^3).
\end{eqnarray} Then we have that
\begin{eqnarray} \label{A1/2u}
    \nonumber
    \|\nabla u\|_{L^2(\mathbb R^3_+)}^2=\|A^{\frac12} u\|_{L^2(\mathbb R^3_+)}^2
    &=&\int_0^{+\infty}\lambda{\rm d}\|E_\lambda u\|_{L^2}^2
    \geq \nonumber
    \rho\int_\rho^{+\infty}{\rm d}\|E_\lambda u\|_{L^2}^2
    \\
    &=& \nonumber
    \rho\int_0^{+\infty}{\rm d}\|E_\lambda u\|_{L^2}^2-\rho\int_0^{\rho}{\rm d}\|E_\lambda u\|_{L^2}^2
    \\
    &=&\rho\big(\|u\|_{L^2(\mathbb R^3_+)}^2-\|E_\rho u\|_{L^2(\mathbb R^3_+)}^2\big).
\end{eqnarray}

\medskip
Motivated by \cite{Han1}, \cite{Han2} and \cite{Han3}, we perform a decomposition of $\mathbb{P}$ in
$L^1(\mathbb{R}^3_+,\mathbb R^3)$ as follows.
For any $f:\mathbb R_+^3\mapsto \mathbb R$, let $q:\mathbb R^3\mapsto\mathbb R$
solve
\begin{eqnarray}
\begin{cases}
    -\Delta q=f\ {\rm in}\ \mathbb{R}^3_+,\\
    \frac{\partial q}{\partial x_3}=0\ {\rm on}\ \partial \mathbb{R}^3_+.
\end{cases}
\end{eqnarray}
Then $q$ can be represented by
\begin{eqnarray}
     q=\mathcal{N}f,\ \mathcal{N}=\int_0^\infty \mathcal{F}(\tau){\rm d}\tau,
\end{eqnarray}
where the operator $\mathcal{F}$ is defined by
\begin{eqnarray}
    (\mathcal{F}(t)f)(x)=\int_{\mathbb{R}^3_+}(G_t((x'-y',x_3-y_3),t)
    +G_t((x'-y',x_3+y_3),t)f(y){\rm d}y,
\end{eqnarray}
and $G_t(x,t)=(4\pi t)^{-\frac32}e^{-\frac{|x|^2}{4t}}$ is the heat kernel
in $\mathbb R^3$. Then we have
\begin{eqnarray}
   \nonumber \label{decomposition}
   &&\mathbb{P}\left(u\cdot\nabla u+\nabla\cdot(\nabla d\odot\nabla d)\right)
   \\&=&
   \left(u\cdot\nabla u+\nabla\cdot(\nabla d\odot\nabla d)\right)+\sum\limits_{i,j=1}^3\nabla\mathcal{N}\partial_i\partial_j\left(u_iu_j
   +\langle\partial_id,\partial_jd\rangle\right).
\end{eqnarray}
We now need to show the following estimate: for any $0<t<\infty$,
\begin{eqnarray}
    \label{P}
    &&\big\|\sum\limits_{i,j=1}^3\nabla\mathcal{N}\partial_i\partial_j(u_iu_j+\langle\partial_id,\partial_jd\rangle)(t)\big\|_{L^1(\mathbb R^3_+)}\nonumber\\
    &&\leq
    C\big(\|u(t)\|_{H^1(\mathbb R^3_+)}^2+\|\nabla d(t)\|_{H^1(\mathbb R^3_+)}^2+\||\nabla d(t)||\nabla^3d(t)|\|_{L^1(\mathbb R^3_+)}\big).
\end{eqnarray}
In fact, for $1\le m\le 3$,
\begin{eqnarray}
    \nonumber
    &&\big\|\sum\limits_{i,j=1}^3\partial_m\mathcal{N}\partial_i\partial_j
    \big(u_iu_j+\langle\partial_id, \partial_jd\rangle\big)(t)\big\|_{L^1(\mathbb R^3_+)}
    \\&\leq&\nonumber \big\|\sum\limits_{i,j=1}^3\partial_m\int_0^\infty\mathcal{F}(\tau)\partial_i\partial_j\big(u_iu_j+\langle\partial_id,
    \partial_jd\rangle\big)(t){\rm d}\tau\big\|_{L^1(\mathbb R^3_+)}
    \\&\leq&\nonumber
    \big\|\sum\limits_{i,j=1}^3\partial_m\big(\int_0^1+\int_1^\infty\big)G_\tau\ast\big(\partial_i\partial_j\left(u_iu_j+\langle\partial_id,\partial_jd\rangle\big)\big)_\ast(t){\rm d}\tau\right\|_{L^1(\mathbb R^3_+)}
    \\&\leq&\nonumber
    \big\|\sum\limits_{i,j=1}^3\partial_m\int_0^1G_\tau\ast\big(\partial_i\partial_j\big(u_iu_j+\langle\partial_id,\partial_jd\rangle\big)\big)_\ast(t){\rm d}\tau\big\|_{L^1(\mathbb R^3_+)}
    \\&&\nonumber
    +\big\|\sum\limits_{i,j=1}^3\partial_m\partial_i\partial_j\int_1^\infty G_\tau\ast\big(u_iu_j+\langle\partial_id, \partial_jd\rangle\big)_\ast(t){\rm d}\tau\big\|_{L^1(\mathbb R^3_+)}
    \\&\leq&\nonumber
    \sum\limits_{i,j=1}^3\int_0^1\|\partial_m G_\tau\|_{L^1(\mathbb R^3)}{\rm d}\tau\big\|\partial_i\partial_j\big(u_iu_j+\langle\partial_id, \partial_jd\rangle\big)(t)\big\|_{L^1(\mathbb R^3_+)}
    \\&&\nonumber
    +\sum\limits_{i,j=1}^3\int_1^\infty\|\partial_m\partial_i\partial_j G_\tau\|_{L^1(\mathbb R^3)}{\rm d}\tau\big\|\big(u_iu_j+\langle\partial_id,
    \partial_jd\rangle\big)(t)\big\|_{L^1(\mathbb R^3_+)}
        \\
        &\leq&\nonumber
    \sum\limits_{i,j=1}^3(\int_0^1\tau^{-\frac12}{\rm d}\tau)\|\partial_m G_1\|_{L^1(\mathbb R^3)}\big\|\partial_i\partial_j\big(u_iu_j+
    \langle\partial_id, \partial_jd\rangle\big)(t)\big\|_{L^1(\mathbb R^3_+)}
    \\&&\nonumber
    +\sum\limits_{i,j=1}^3(\int_1^\infty\tau^{-\frac32}{\rm d}\tau)
    \|\partial_m\partial_i\partial_j G_1\|_{L^1(\mathbb R^3)}
    \big\|\big(u_iu_j+\langle\partial_id, \partial_jd\rangle\big)(t)\big\|_{L^1(\mathbb R^3_+)}
    \\&\leq&
    C\big(\|u(t)\|_{H^1(\mathbb R^3+)}^2+
    \|\nabla d(t)\|_{H^1(\mathbb R^3_+)}^2
    +\||\nabla d(t)||\nabla^3d(t)|\|_{L^1(\mathbb R^3_+)}\big),
\end{eqnarray}
where $f*g$ represents the convolution of $f$ and $g$,
and $f_*$ represents the even extension of function $f$ with respect to $x_3$
 from $\mathbb{R}^3_+$ to $\mathbb{R}^3$.

Next we need the following estimate (see also \cite{Bocher2}).
\begin{lemma} \label{key-estimate} For any $\alpha_1\in (0,1)$ and $\alpha_2,\alpha_3>0$, it holds
\begin{eqnarray} \label{s}
    \begin{cases}
   \displaystyle \int_0^t(t-s)^{-\alpha_1}s^{-\alpha_2}{\rm d}s\leq
    Ct^{1-\alpha_1-\alpha_2},\ {\rm for}\ 0<\alpha_2<1,\\[3mm]
    \displaystyle\int_{\frac{t}{2}}^t(t-s)^{-\alpha_1}s^{-\alpha_2}{\rm d}s\leq
    Ct^{1-\alpha_1-\alpha_2},\ {\rm for}\ \alpha_2>0,
    \end{cases}
\end{eqnarray}
\begin{eqnarray}
    \int_0^t(t-s)^{-\alpha_1}(1+s)^{-\alpha_3}{\rm d}s\leq
    \begin{cases}
    Ct^{-\alpha_1},\ {\rm for}\ \alpha_3>1,\\
    Ct^{1-\alpha_1-\alpha_3},\ {\rm for}\ 0<\alpha_3<1,\\
    Ct^{-\alpha_1}\ln (1+t),\ {\rm for}\ \alpha_3=1,
    \end{cases}
\end{eqnarray} and
\begin{eqnarray}
    \int_{\frac{t}{2}}^t(t-s)^{-\alpha_1}(1+s)^{-\alpha_3}{\rm d}s\leq
    Ct^{1-\alpha_1-\alpha_3}\ {\rm for}\ \alpha_3>0.
\end{eqnarray}
\end{lemma}
\begin{proof} For convenience of readers,
we sketch a proof. By direct calculations, we have
that for $\alpha_2\in(0,1)$,
\begin{eqnarray}
    \nonumber
    \int_0^t(t-s)^{-\alpha_1}s^{-\alpha_2}{\rm d}s
    &=&\nonumber
    \int_0^{\frac{t}{2}}(t-s)^{-\alpha_1}s^{-\alpha_2}{\rm d}s+\int^t_{\frac{t}{2}}(t-s)^{-\alpha_1}s^{-\alpha_2}{\rm d}s
    \\&\leq&\nonumber
    \big(\frac{t}{2}\big)^{-\alpha_1}\frac{s^{1-\alpha_2}}{1-\alpha_2}\Big|_0^{\frac{t}{2}} +\big(\frac{t}{2}\big)^{-\alpha_2}\frac{-(t-s)^{1-\alpha_1}}{1-\alpha_1}\Big|^t_{\frac{t}{2}}
    \\&\leq&
    Ct^{1-\alpha_1-\alpha_2}.\nonumber
\end{eqnarray} This
yields (\ref{s}).

Similarly, we have
\begin{eqnarray}
    \nonumber
    \int_0^t(t-s)^{-\alpha_1}(1+s)^{-\alpha_3}{\rm d}s
    &=&\nonumber
    \int_0^{\frac{t}{2}}(t-s)^{-\alpha_1}(1+s)^{-\alpha_3}{\rm d}s+\int^t_{\frac{t}{2}}(t-s)^{-\alpha_1}(1+s)^{-\alpha_3}{\rm d}s
    \\&\leq&\nonumber
    \begin{cases}
    \left(\frac{t}{2}\right)^{-\alpha_1}\frac{(1+s)^{1-\alpha_3}}{1-\alpha_3}\Big|_0^{\frac{t}{2}} +\left(1+\frac{t}{2}\right)^{-\alpha_3}\frac{-(t-s)^{1-\alpha_1}}{1-\alpha_1}\Big|^t_{\frac{t}{2}},\ {\rm for}\ \alpha_3\not=1\\
    \left(\frac{t}{2}\right)^{-\alpha_1}\ln(1+s)\Big|_0^{\frac{t}{2}} +\left(1+\frac{t}{2}\right)^{-\alpha_3}\frac{-(t-s)^{1-\alpha_1}}{1-\alpha_1}\Big|^t_{\frac{t}{2}},
    \ {\rm for}\ \alpha_3=1
    \end{cases}
    \\&\leq&
    \begin{cases}
    Ct^{-\alpha_1}, \ {\rm for}\ \alpha_3>1,\\
    Ct^{1-\alpha_1-\alpha_3}, \ {\rm for}\ \alpha_3<1,\\
    Ct^{-\alpha_1}\ln (1+t), \ {\rm for}\ \alpha_3=1.
    \end{cases}\nonumber
\end{eqnarray} This completes the proof.
\end{proof}

\section{Proof of Theorem \ref{th:E_0}}

This section will be devoted to the proof of Theorem \ref{th:E_0}. It is divided into several subsections
and several Lemmas.
\subsection{Global existence of strong solutions}
The  local existence of strong solutions as stated in Theorem
\ref{th:E_0} can be established by the same approach as
\cite{Hineman}, which is omitted.
To show the time interval can be extended globally
and to establish the optimal time decay rates, presented in subsection \ref{sub3.3}
below, we need to obtain a few {\it a priori} estimates.

Assume $(u,d,p)$ is a strong solution to
(\ref{system})--(\ref{i-b condition}) in
$\mathbb{R}^3_+\times[0,T]$ for some $T>0$,
we have

\begin{lemma} \label{le:B} For any $t\in [0,T]$, it
holds that
\begin{eqnarray} \label{basci energy}
    &&\frac{{\rm d}}{{\rm d}t}\big(\|u(t)\|_{L^2(\mathbb R^3_+)}^2
    +\|\nabla d(t)\|_{L^2(\mathbb R^3_+)}^2\big)\nonumber\\
    &&=-
    2\big(\|\nabla u(t)\|_{L^2(\mathbb R^3_+)}^2
    +\|(\Delta d+|\nabla d|^2d(t)\|_{L^2(\mathbb R^3_+)}^2\big).
\end{eqnarray} In particular, it holds that
$$\mathbb{E}(t)=\|u(t)\|_{L^2(\mathbb R^3_+)}^2
+\|\nabla d(t)\|_{L^2(\mathbb R^3_+)}^2
\leq \mathbb{E}(0)=\|u_0\|_{L^2(\mathbb R^3_+)}^2+\|\nabla d_0\|_{L^2(\mathbb R^3_+)}^2.
$$
\end{lemma}
\begin{proof} Multiplying (\ref{system})$_1$ by $u$
and (\ref{system})$_3$ by $(\Delta d+|\nabla d|^2d)$, adding and integrating the resulting equations over $\mathbb R^3_+$,
applying (\ref{system})$_2$, the fact that $|d|=1$, and integration by parts, we can obtain (\ref{basci energy}).
\end{proof}

\begin{lemma} \label{le:L3} There exists $C>0$, independent of $T$, such that
for any $0<t<T$, it holds
\begin{eqnarray}
   &&\nonumber
   \frac{{\rm d}}{{\rm d}t}\int_{\mathbb R^3_+}\big(|u(t)|^3+|\nabla d(t)|^3\big){\rm d}x+\big[1-C\|u(t)\|_{L^3(\mathbb R^3_+)}^2\big]\int_{\mathbb R^3_+}|u(t)||\nabla u(t)|^2{\rm d}x
   \\
   &&
   +\big[1-C\big(\|u(t)\|_{L^3(\mathbb R^3_+)}+\|u(t)\|_{L^3(\mathbb R^3_+)}\|\nabla d(t)\|_{L^3(\mathbb R^3_+)}+\|\nabla d(t)\|_{L^3(\mathbb R^3_+)}^2\big)\big]\nonumber\\
   &&\ \ \cdot\int_{\mathbb R^3_+}|\nabla d(t)||\nabla^2 d(t)|^2{\rm d}x\leq 0.
\end{eqnarray}
In particular, there exists $\varepsilon_0>0$ such that if
$$\|u_0\|_{L^3(\mathbb R^3_+)}^3+\|\nabla d_0\|_{L^3(\mathbb R^3_+)}^3
\le\varepsilon_0^3,$$
then for any $0<t<T$,
$$\|u(t)\|_{L^3(\mathbb R^3_+)}^3+\|\nabla d(t)\|_{L^3(\mathbb R^3_+)}^3\leq
\|u_0\|_{L^3(\mathbb R^3_+)}^3+\|\nabla d_0\|_{L^3(\mathbb R^3_+)}^3.
$$
\end{lemma}
\begin{proof} We closely follow the proof of \cite{Hineman} Lemma 2.1.
In contrast with \cite{Hineman} Lemma 2.1, we need to verify that boundary
contributions are zero in the process of integration by parts. Take derivative
of \eqref{system}$_3$, multiply the resulting equation by $|\nabla d|\nabla d$, and integrate over $\mathbb R^3_+$, we can check that, for example,
there is no boundary contribution from the following term.
\begin{eqnarray}
    &&\nonumber\
    \int_{\mathbb R^3_+} \nabla\Delta d:|\nabla d|\nabla d{\rm d}x
    \\&=&
   -\underbrace{\int_{\partial\mathbb R^3_+}\langle\frac{\partial^2 d}
   {\partial x_j\partial x_3}, |\nabla d|\frac{\partial d}{\partial x_j}\rangle
   {\rm d}\sigma}_{J}-{\int_{\mathbb R^3_+\cap\{|\nabla d|>0\}}\nabla^2d:\nabla\left(|\nabla d|\nabla d\right){\rm d}x}.
\end{eqnarray}
Since $\frac{\partial d}{\partial x_3}=0$ on $\partial\mathbb R^3_+$,
we have that $\frac{\partial^2 d}{\partial x_j\partial x_3}=0$, for
$j=1,2$, on
$\partial\mathbb R^3_+$ and hence
\begin{eqnarray}
    \nonumber \label{nor-tan}
    J&=&\sum_{j=1}^2\int_{\partial\mathbb R^3_+}
    \langle\frac{\partial^2 d}{\partial x_j\partial x_3}, |\nabla d|
    \frac{\partial d}{\partial x_j}\rangle\,d\sigma=0.
\end{eqnarray}
This yields that
\begin{eqnarray}
    \int_{\mathbb R_+^3} \nabla\Delta d:|\nabla d|\nabla d{\rm d}x
    &=&\nonumber
    -\int_{\mathbb R^3_+\cap\{|\nabla d|>0\}}\big(|\nabla d||\nabla^2d|^2+\frac{|\nabla^2d\cdot\nabla d|^2}{|\nabla d|}\big){\rm d}x
    \\&\leq&
    -\int_{\mathbb R^3_+\cap\{|\nabla d|>0\}}|\nabla d||\nabla^2d|^2{\rm d}x.
\end{eqnarray}
The remaining parts of proof follow \cite{Hineman} Lemma 2.1 line by line,
which is omitted.
\end{proof}

\begin{lemma} \label{le:3.7} There exists $C>0$, independent of $T$, such that
for any $t\in [0,T]$, it holds that
\begin{eqnarray}
   &&\nonumber
   \sigma(t)\big(\|\nabla u(t)\|_{L^2(\mathbb R^3_+)}^2+\|d_t(t)\|_{L^2(\mathbb R^3_+)}^2+\|\nabla^2d(t)\|_{L^2(\mathbb R^3_+)}^2\big)+
   \\&&
   \int_0^t\sigma(s)\big(\|u_t(s)\|_{L^2(\mathbb R^3_+)}^2
   +\|\nabla u(s)\|_{H^1(\mathbb R^3_+)}^2+\|d_t(s)\|_{H^1(\mathbb R^3_+)}^2+\|\nabla^2 d(s)\|_{H^1(\mathbb R^3_+)}^2\big){\rm d}s\leq C,
\end{eqnarray}
where $\sigma(t)=\min\{1, t\}$.
\end{lemma}
\begin{proof} Multiplying (\ref{system})$_1$ by $u_t$ and integrating over
$\mathbb R^3_+$, and applying the interpolation inequality and the Sobolev inequality, we obtain
\begin{eqnarray}
  && \nonumber
   \frac12\frac{{\rm d}}{{\rm d}t}\int_{\mathbb R^3_+}|\nabla u|^2{\rm d}x+\int_{\mathbb R^3_+}|u_t|^2
   \leq\nonumber
   \int_{\mathbb R^3_+}\big(|u||\nabla u||u_t|+|\nabla d||\nabla^2d||u_t|\big)
   {\rm d}x
   \\&\leq&\nonumber
   C\|u\|_{L^3(\mathbb R^3_+)}\|\nabla u\|_{L^6(\mathbb R^3_+)}\|u_t\|_{L^2(\mathbb R^3_+)}+C\|\nabla d\|_{L^3(\mathbb R^3_+)}\|\nabla^2d\|_{L^6(\mathbb R^3_+)}\|u_t\|_{L^2(\mathbb R^3_+)}
   \\&\leq&
   C\|u\|_{L^3(\mathbb R^3_+)}\|\nabla u\|_{H^1(\mathbb R^3_+)}\|u_t\|_{L^2(\mathbb R^3_+)}+C\|\nabla d\|_{L^3(\mathbb R^3_+)}\|\nabla^2d\|_{H^1(\mathbb R^3_+)}\|u_t\|_{L^2(\mathbb R^3_+)}.
\end{eqnarray}
By the standard estimates of the Stokes equation on $\mathbb{R}^3_+$ (see \cite{Galdi}), we obtain
\begin{eqnarray}
     \nonumber
     \|\nabla^2 u\|_{L^2(\mathbb R^3_+)}&\leq& C\big(\|u\cdot\nabla u\|_{L^2(\mathbb R^3_+)}+\|\nabla\cdot(\nabla d\odot\nabla d)\|_{L^2(\mathbb R^3_+)}+\|u_t\|_{L^2(\mathbb R^3_+)}\big)
     \\&\leq&\nonumber
     C\big(\|u\|_{L^3(\mathbb R^3_+)}\|\nabla u\|_{L^6(\mathbb R^3_+)}+\|\nabla d\|_{L^3(\mathbb R^3_+)}\|\nabla^2 d\|_{L^6(\mathbb R^3_+)}+\|u_t\|_{L^2(\mathbb R^3_+)}\big)
     \\&\leq&
     C\big(\|u\|_{L^3(\mathbb R^3_+)}\|\nabla u\|_{H^1(\mathbb R^3_+)}+\|\nabla d\|_{L^3(\mathbb R^3_+)}\|\nabla^2 d\|_{H^1(\mathbb R^3_+)}+\|u_t\|_{L^2(\mathbb R^3_+)}\big).
\end{eqnarray}
Combining the above two inequalities, we get
\begin{eqnarray}
   \nonumber\label{nabel u}
   &&\frac{{\rm d}}{{\rm d}t}\|\nabla u\|_{L^2(\mathbb R^3_+)}^2
   +\big(\|u_t\|_{L^2(\mathbb R^3_+)}^2+\|\nabla u\|_{H^1(\mathbb R^3_+)}^2\big)
   \\&&\leq\nonumber
   C\big(\|u\|_{L^3(\mathbb R^3_+)}+\|u\|_{L^3(\mathbb R^3_+)}^2+\|\nabla d\|_{L^3(\mathbb R^3_+)}+\|\nabla d\|_{L^3(\mathbb R^3_+)}^2\big)\nonumber\\
   && \ \ \ \ \ \ \cdot\big(\|\nabla u\|_{H^1(\mathbb R^3_+)}^2+\|u_t\|_{L^2(\mathbb R^3_+)}^2+\|\nabla^2 d\|_{H^1(\mathbb R^3_+)}\big)
+C\|\nabla u\|_{L^2(\mathbb R^3_+)}^2.
\end{eqnarray}

Next, taking $\partial_t$ of (\ref{system})$_3$, we have
\begin{eqnarray}\label{dtt}
   d_{tt}+u_t\cdot\nabla d+u\cdot\nabla d_t=\Delta d_t+2\langle\nabla d,\nabla d_t\rangle d+|\nabla d|^2d_t.
\end{eqnarray}
Multiplying \eqref{dtt} by $d_t$, integrating over $\mathbb R^3_+$,
applying $\frac{\partial d_t}{\partial x_3}=0$
on $\partial\mathbb R^3_+$,  (\ref{system})$_2$, the fact $|d|=1$, and the interpolation inequality
and the Sobolev inequality, we obtain
\begin{eqnarray}
   \nonumber
   &&\frac12\frac{{\rm d}}{{\rm d}t}\int_{\mathbb R^3_+}|d_t|^2{\rm d}x
   +\int_{\mathbb R^3_+}|\nabla d_t|^2{\rm d}x
   \\&\leq&\nonumber
   C\int_{\mathbb R^3_+}\left(|u_t||\nabla d||d_t|+|\nabla d|^2|d_t|^2\right){\rm d}x
   \\&\leq&\nonumber
   C\|u_t\|_{L^2(\mathbb R^3_+)}\|\nabla d\|_{L^3(\mathbb R^3_+)}\|d_t\|_{L^6(\mathbb R^3_+)}+\|\nabla d\|_{L^3(\mathbb R^3_+)}^2\|d_t\|_{L^6(\mathbb R^3_+)}^2
   \\&\leq&\nonumber
   C\|u_t\|_{L^2(\mathbb R^3_+)}\|\nabla d\|_{L^3(\mathbb R^3_+)}\|d_t\|_{H^1(\mathbb R^3_+)}+\|\nabla d\|_{L^3(\mathbb R^3_+)}^2\|d_t\|_{H^1(\mathbb R^3_+)}^2
   \\&\leq&
   C\big(\|\nabla d\|_{L^3(\mathbb R^3_+)}+\|\nabla
   d\|_{L^3(\mathbb R^3_+)}^2\big)\big(\|u_t\|_{L^2(\mathbb R^3_+)}^2+\|d_t\|_{H^1(\mathbb R^3_+)}^2\big).
\end{eqnarray}
By the elliptic estimate, we have
\begin{eqnarray}
    \nonumber
    &&\|\nabla^3 d\|_{L^2(\mathbb R^3_+)}
    \\&\leq& \nonumber
    C\big(\|\nabla d_t\|_{L^2(\mathbb R^3_+)}+\|\nabla\left(u\cdot\nabla d\right)\|_{L^2(\mathbb R^3_+)}+\|\nabla(|\nabla d|^2d)\|_{L^2(\mathbb R^3_+)}\big)
    \\&\leq&\nonumber
    C\big(\|\nabla d_t\|_{L^2(\mathbb R^3_+)}+\|\nabla u\|_{L^6(\mathbb R^3_+)}\|\nabla d\|_{L^3(\mathbb R^3_+)}\nonumber\\
    &&\quad+\|u\|_{L^3(\mathbb R^3_+)}\|\nabla^2d\|_{L^6(\mathbb R^3_+)}+\|\nabla d\|_{L^3(\mathbb R^3_+)}\|\nabla^2d\|_{L^6(\mathbb R^3_+)}\big)
    \nonumber\\
    &\leq&
    C\big[\|\nabla d_t\|_{L^2(\mathbb R^3_+)}+\big(\|u\|_{L^3(\mathbb R^3_+)}+\|\nabla d\|_{L^3(\mathbb R^3_+)}\big)\big(\|\nabla u\|_{H^1(\mathbb R^3_+)}+\|\nabla^2d\|_{H^1(\mathbb R^3_+)}\big)\big],
\end{eqnarray}
where we have used the fact that
$|\nabla d|^2=-d\cdot\Delta d\le |\Delta d|$ in the second inequality.
Combining these two inequalities and using the Cauchy inequality, we obtain
\begin{eqnarray}
   \nonumber\label{ut}
   &&\frac{{\rm d}}{{\rm d}t}\|d_t\|_{L^2(\mathbb R^3_+)}^2
   +\big(\|d_t\|_{H^1(\mathbb R^3_+)}^2+\|\nabla^2 d\|_{H^1(\mathbb R^3_+)}^2\big)
   \\&&\leq\nonumber
   C\big(\|u\|_{L^3(\mathbb R^3_+)}^2+\|\nabla d\|_{L^3(\mathbb R^3_+)}+\|\nabla d\|_{L^3(\mathbb R^3_+)}^2\big)\nonumber\\
   &&\quad\cdot\big(\|\nabla u\|_{H^1(\mathbb R^3_+)}^2+\|u_t\|_{L^2(\mathbb R^3_+)}^2+\|d_t\|_{H^1(\mathbb R^3_+)}^2+\|\nabla^2 d\|_{H^1(\mathbb R^3_+)}^2\big)
   \nonumber\\
   &&\quad+C\left(\|d_t\|_{L^2(\mathbb R^3_+)}^2+\|\nabla^2d\|_{L^2(\mathbb R^3_+)}^2\right).
\end{eqnarray}
Combining (\ref{nabel u}) and (\ref{ut}), we have
\begin{eqnarray} \label{nabelu+dt}
   &&\nonumber
   \frac{{\rm d}}{{\rm d}t}\big(\|\nabla u\|_{L^2(\mathbb R^3_+)}^2+\|d_t\|_{L^2(\mathbb R^3_+)}^2\big)\nonumber\\
   &&\quad+\big(1-C\mathcal{G}(t)\big)
   \big(\|u_t\|_{L^2(\mathbb R^3_+)}^2+\|\nabla u\|_{H^1(\mathbb R^3_+)}^2+\|d_t\|_{H^1(\mathbb R^3_+)}^2+\|\nabla^2 d\|_{H^1(\mathbb R^3_+)}^2\big)
  \nonumber \\
  &&\leq
   C\big(\|\nabla u\|_{L^2(\mathbb R^3_+)}^2+\|d_t\|_{L^2(\mathbb R^3_+)}^2+\|\nabla^2d\|_{L^2(\mathbb R^3_+)}^2\big).
\end{eqnarray} where
$$\mathcal{G}(t)=\big(\|u(t)\|_{L^3(\mathbb R^3_+)}+\|u(t)\|_{L^3(\mathbb R^3_+)}^2+\|\nabla d(t)\|_{L^3(\mathbb R^3_+)}+\|\nabla d(t)\|_{L^3(\mathbb R^3_+)}^2\big).$$

On the other hand, let $\sigma(t)=\min\{1,t\}$. Then we have
\begin{eqnarray}\label{314}
    &&\nonumber
    \frac{{\rm d}}{{\rm d}t}\big[\sigma(t)\int_{\mathbb R^3_+}\big(|\nabla u|^2+|d_t|^2\big){\rm d}x\big]
    \\&\le&
    \int_{\mathbb R^3_+}\big(|\nabla u|^2+|d_t|^2\big){\rm d}x+\sigma(t)\frac{{\rm d}}{{\rm d}t}\int_{\mathbb R^3_+}\big(|\nabla u|^2+|d_t|^2\big){\rm d}x.
\end{eqnarray}
Multiplying (\ref{nabelu+dt}) by $\sigma(t)$ and substituting it into
\eqref{314}, we have
\begin{eqnarray}
   &&\nonumber\label{nabel-u+ut}
   \frac{{\rm d}}{{\rm d}t}\big[\sigma(t)\big(\|\nabla u\|_{L^2(\mathbb R^3_+)}^2+\|d_t\|_{L^2(\mathbb R^3_+)}^2\big)\big]
   \\&&\nonumber
   +\big(1-C\mathcal{G}(t)\big)\sigma(t)\big(\|u_t\|_{L^2(\mathbb R^3_+)}^2+\|\nabla u\|_{H^1(\mathbb R^3_+)}^2+\|d_t\|_{H^1(\mathbb R^3_+)}^2+\|\nabla^2 d\|_{H^1(\mathbb R^3_+)}^2\big)
   \\&&\leq
   C\big(\|\nabla u\|_{L^2(\mathbb R^3_+)}^2+\|d_t\|_{L^2(\mathbb R^3_+)}^2+\|\nabla^2d\|_{L^2(\mathbb R^3_+)}^2\big).
\end{eqnarray}
Observe that
\begin{eqnarray} \label{d-t}
    &&\int_0^t\|d_t(s)\|_{L^2(\mathbb R^3_+)}^2{\rm d}s\nonumber\\
    &&\leq C\int_0^t\big(\|(d_t+u\cdot\nabla d)(s)\|_{L^2(\mathbb R^3_+)}^2
    +\|\nabla u(s)\|_{L^2(\mathbb R^3_+)}^2\|\nabla d(s)\|_{L^3(\mathbb R^3_+)}^2\big){\rm d}s\nonumber\\
    &&\leq C.
\end{eqnarray}
This, together with
$\displaystyle\int_0^t\|\nabla u(s)\|_{L^2(\mathbb R^3_+)}^2{\rm d}s\leq C$, implies
that there exists $t_k\rightarrow 0$ such that
\begin{eqnarray}
   t_k\big(\|\nabla u(t_k)\|^2_{L^2(\mathbb R^3_+)}+\|d_t(t_k)\|_{L^2(\mathbb R^3_+)}^2\big)\rightarrow 0.
\end{eqnarray}
By the elliptic estimates, we have
\begin{eqnarray} \label{nabla-2d}
    \nonumber\|\nabla^2d\|_{L^2(\mathbb R^3_+)}&\leq& C\|d_t+u\cdot\nabla d\|_{L^2(\mathbb R^3_+)}+C\|\nabla d\|_{L^4(\mathbb R^3_+)}^2
    \\&\leq&\nonumber
    C\|d_t+u\cdot\nabla d\|_{L^2(\mathbb R^3_+)}
    +C\|\nabla d\|_{L^3(\mathbb R^3_+)}\|\nabla d\|_{L^6(\mathbb R^3_+)}
    \\&\leq&
    C\|d_t+u\cdot\nabla d\|_{L^2(\mathbb R^3_+)}+C\varepsilon_0\|\nabla^2 d\|_{L^2(\mathbb R^3_+)},
\end{eqnarray}
where we have used the Sobolev inequality \footnote{in fact, since $\frac{\partial d}{\partial x_3}=0$ on $\partial\mathbb R^3_+$, this follows from an even extension of $d$ from $\mathbb R^3_+$ to $\mathbb R^3$. See also \cite{Bocher}}
$$\|\nabla d\|_{L^6(\mathbb R^3_+)}\leq C\|\nabla^2 d\|_{L^2(\mathbb R^3_+)}.
$$
Thus if we choose a sufficiently small $\varepsilon_0>0$,
then we obtain that
\begin{equation}\label{na2dL2}\int_0^t\|\nabla^2d(s)\|^2_{L^2(\mathbb R^3_+)}{\rm d}s\leq C.
\end{equation}

By integrating (\ref{nabel-u+ut}) over $[t_k,t]$ and sending $k$ to
$\infty$, we finally obtain
\begin{eqnarray}
&&\nonumber \label{3.19}
\sigma(t)\big(\|\nabla u(t)\|_{L^2(\mathbb R^3_+)}^2+\|d_t\|_{L^2(\mathbb R^3_+)}^2\big)+\\
&&\int_0^t\sigma(s)\big(\|u_t(s)\|_{L^2(\mathbb R^3_+)}^2+\|\nabla u(s)\|_{H^1(\mathbb R^3_+)}^2+\|d_t(s)\|_{H^1(\mathbb R^3_+)}^2
   +\|\nabla^2 d(s)\|_{H^1(\mathbb R^3_+)}^2\big){\rm d}s\leq C.
\end{eqnarray}
By the elliptic estimate \eqref{nabla-2d}, we have
\begin{eqnarray*}
\|\nabla^2d\|_{L^2(\mathbb R^3_+)}\leq
    C\|d_t\|_{L^2(\mathbb R^3_+)}+C\varepsilon_0\|\nabla^2 d\|_{L^2(\mathbb R^3_+)}.
\end{eqnarray*}
By choosing a sufficiently small $\varepsilon_0>0$,
this yields
\begin{equation}
\label{nabla.2d}
\|\nabla^2d\|_{L^2(\mathbb R^3_+)}\leq C\|d_t\|_{L^2(\mathbb R^3_+)}.
\end{equation}
 Hence we obtain
\begin{eqnarray}
   &&\nonumber \label{3190}
   \sigma(t)\big(\|\nabla u(t)\|_{L^2(\mathbb R^3_+)}^2+\|d_t(t)\|_{L^2(\mathbb R^3_+)}^2+\|\nabla^2d(t)\|_{L^2(\mathbb R^3_+)}^2\big)
   \\&&
   +\int_0^t\sigma(s)\big(\|u_t(s)\|_{L^2(\mathbb R^3_+)}^2+\|\nabla u(s)\|_{H^1(\mathbb R^3_+)}^2+\|d_t(s)\|_{H^1(\mathbb R^3_+)}^2+\|\nabla^2 d(s)\|_{H^1(\mathbb R^3_+)}^2\big){\rm d}s\leq C.
\end{eqnarray}
This completes the proof.
\end{proof}

It is clear that (\ref{3190}) implies that for any small $\tau>0$,
there exists a positive constant $C_\tau$ such that for $\tau<t<T$,
\begin{eqnarray} \label{3.200}
   &&\nonumber
   \big(\|\nabla u(\tau)\|_{L^2(\mathbb R^3_+)}^2+\|\nabla^2d(\tau)\|_{L^2(\mathbb R^3_+)}+\|d_t(\tau)\|_{L^2(\mathbb R^3_+)}^2\big)
   \\&&
   +\int_\tau^t\big(\|u_t(s)\|_{L^2(\mathbb R^3_+)}^2+\|\nabla u(s)\|_{H^1(\mathbb R^3_+)}^2+\|d_t(s)\|_{H^1(\mathbb R^3_+)}^2+\|\nabla^2 d(s)\|_{H^1(\mathbb R^3_+)}^2\big){\rm d}s\leq C_\tau.
\end{eqnarray}
Theorem \ref{th:E_0} follows from higher order elliptic estimates.
More precisely, we have the following Lemma.

\begin{lemma} \label{le:higher order} For any $0<\tau<t<T$,
there exists a positive constant $C_\tau$ such that
\begin{eqnarray}\label{3180}
   &&\nonumber
   \big(\|\nabla^2 u(\tau)\|_{L^2(\mathbb R^3_+)}^2+\|u_t(\tau)\|_{L^2(\mathbb R^3_+)}^2
   +\|\nabla^3 d(\tau)\|_{L^2(\mathbb R^3_+)}^2+\|\nabla d_t(\tau)\|_{L^2(\mathbb R^3_+)}^2\big)+
   \\&&
   \int_\tau^t\big(\|\nabla u_t(s)\|_{L^2(\mathbb R^3_+)}^2+\|\nabla^3 u(s)\|_{L^2(\mathbb R^3_+)}^2+\|\nabla^2 d_t(s)\|_{L^2(\mathbb R^3_+)}^2+\|\nabla^4 d(s)\|_{L^2(\mathbb R^3_+)}^2\big){\rm d}s\leq C_\tau.
\end{eqnarray}
\end{lemma}
\begin{proof} First, taking $\partial_t$ of (\ref{system})$_1$,  multiplying the resulting equations by $u_t$, and integrating over $\mathbb{R}^3_+$,
we have
\begin{eqnarray}
     \frac12\frac{{\rm d}}{{\rm d}t}\|u_t\|_{L^2(\mathbb R^3_+)}^2+\|\nabla u_t\|_{L^2(\mathbb R^3_+)}^2\leq \int_{\mathbb{R}^3_+}\big(|u_t||u||\nabla u_t|+|\nabla d_t||\nabla d||\nabla u_t|\big){\rm d}x,
\end{eqnarray}
where we have used
$$\int_{\mathbb{R}^3_+}(u_t\cdot\nabla)u\cdot u_t{\rm d}x
=-\int_{\mathbb R^3_+} (u_t\cdot \nabla) u_t \cdot u.$$
Using the H\"older inequality, the Cauchy inequality and the Sobolev inequality, we get
\begin{eqnarray}
     \nonumber
     \frac12\frac{{\rm d}}{{\rm d}t}\|u_t\|_{L^2(\mathbb R^3_+)}^2+\|\nabla u_t\|_{L^2(\mathbb R^3_+)}^2
     &\leq&
     C\|u\|_{L^3(\mathbb R^3_+)}^2\|u_t\|_{L^6(\mathbb R^3_+)}^2
     +C\|\nabla d\|_{L^3(\mathbb R^3_+)}^2\|\nabla d_t\|_{L^6(\mathbb R^3_+)}^2
     \\&\leq&
     C\|u\|_{L^3(\mathbb R^3_+)}^2\|\nabla u_t\|_{L^2(\mathbb R^3_+)}^2+C\|\nabla d\|_{L^3(\mathbb R^3_+)}^2\|\Delta d_t\|_{L^2(\mathbb R^3_+)}^2,
\end{eqnarray} where we have used the fact that
\begin{eqnarray} \label{second order-d}
\|\nabla^2 d_t\|_{L^2(\mathbb R^3_+)}=\|\Delta d_t\|_{L^2(\mathbb R^3_+)}.
\end{eqnarray}
In fact, by integration by parts and (\ref{i-b condition}), one has
\begin{eqnarray}
    \nonumber
    \int_{\mathbb{R}^3_+}|\Delta d_t|^2{\rm d}x&=&-\int_{\mathbb{R}^3_+}\langle\partial_j\partial_i\partial_id_t, \partial_jd_t\rangle{\rm d}x+\int_{\partial\mathbb R^3_+}\langle\partial_i\partial_id_t,
    \partial_3 d_t\rangle\,d\sigma\\
 &=&-\int_{\mathbb{R}^3_+}\langle\partial_j\partial_i\partial_id_t, \partial_jd_t\rangle{\rm d}x\nonumber\\
&=&\int_{\mathbb{R}^3_+}|\nabla^2 d_t|^2{\rm d}x-\int_{\partial\mathbb R^3_+}(\sum_{j=1}^2\langle\partial_j\partial_3d_t, \partial_jd_t\rangle
+\langle\partial_3^2 d_t, \partial_3 d_t\rangle) {\rm d}\sigma\nonumber\\
&=&\int_{\mathbb{R}^3_+}|\nabla^2 d_t|^2{\rm d}x.\nonumber
\end{eqnarray}

Next, taking $\partial_t$ of (\ref{system})$_3$ and multiplying the resulting equations by $\Delta d_t$, integrating over $\mathbb{R}^3_+$, using integration by parts, (\ref{system})$_2$ and the Cauchy inequality, we have
\begin{eqnarray}
     &&\nonumber\frac12\frac{{\rm d}}{{\rm d}t}
     \|\nabla d_t\|_{L^2(\mathbb{R}^3_+)}^2
     +\|\Delta d_t\|_{L^2(\mathbb R^3_+)}^2
     \\&&\leq \nonumber
     C\int_{\mathbb{R}^3_+}\left(|u_t|^2|\nabla d|^2+|u|^2|\nabla d_t|^2+|\nabla d|^2|\nabla d_t|^2\right){\rm d}x-\int_{\mathbb{R}^3_+}|\nabla d|^2(d_t\cdot\Delta d_t){\rm d}x
     \\&&=
     I_1+I_2.
\end{eqnarray}
By H\"older's inequality, we can estimate $I_1$ by
\begin{eqnarray}
     |I_1|&\leq&\nonumber
     C\|u_t\|_{L^6(\mathbb R^3_+)}^2\|\nabla d\|_{L^3(\mathbb R^3_+)}^2+C\|u\|_{L^3(\mathbb R^3_+)}^2\|\nabla d_t\|_{L^6(\mathbb R^3_+)}^2+C\|\nabla d\|_{L^3(\mathbb R^3_+)}^2\|\nabla d_t\|_{L^6(\mathbb R^3_+)}^2
     \\&\leq&
     C\big(\|u\|_{L^3(\mathbb R^3_+)}^2+\|\nabla d\|_{L^3(\mathbb R^3_+)}^2\big)\big(\|\nabla u_t\|_{L^2(\mathbb R^3_+)}^2+\|\Delta d_t\|_{L^2(\mathbb R^3_+)}^2\big).
\end{eqnarray}
While $I_2$ can be estimated by
\begin{eqnarray}
     |I_2|&\leq&\nonumber
     \int_{\mathbb{R}^3_+}|\nabla d|^2|\nabla d_t|^2{\rm d}x+2\int_{\mathbb{R}^3_+}\langle\nabla d,\nabla^2d\rangle d_t\cdot\nabla d_t{\rm d}x
     \\&\leq&\nonumber
     C\|\nabla d\|_{L^3}^2\|\nabla d_t\|_{L^6}^2+C\|\nabla d\|_{L^6}\|\nabla^2d\|_{L^2}\|d_t\|_{L^6}\|\nabla d_t\|_{L^6}
     \\&\leq&
     \big(C\|\nabla d\|_{L^3}^2+\frac12\big)\|\Delta d_t\|_{L^2(\mathbb R^3_+)}^2+C\|\nabla d\|_{H^1(\mathbb R^3_+)}^4\|\nabla d_t\|_{L^2(\mathbb R^3_+)}^2.
\end{eqnarray}

Putting these two estimates together, we obtain
\begin{eqnarray}\label{3200}
&&\frac{{\rm d}}{{\rm d}t}\big(\|u_t\|_{L^2(\mathbb R^3_+)}^2
+\|\nabla d_t\|_{L^2(\mathbb R^3_+)}^2\big)\nonumber\\
&&+\big(1-C\mathcal{G}(t)\big)\big(\|\nabla u_t\|_{L^2(\mathbb R^3_+)}^2+\|\Delta d_t\|_{L^2(\mathbb R^3_+)}^2\big)\nonumber\\
&&\leq
     C\|\nabla d\|_{H^1(\mathbb R^3_+)}^4\|\nabla d_t\|_{L^2(\mathbb R^3_+)}^2.
\end{eqnarray}
Multiplying \eqref{3200} by $\sigma^3(t)$ and choosing a sufficiently small $\varepsilon_0$, we obtain
\begin{eqnarray}
   &&\nonumber\label{nabell-u+ut}
   \frac{{\rm d}}{{\rm d}t}\big[\sigma^3(t)\big(\|u_t\|_{L^2(\mathbb R^3_+)}^2+\|\nabla d_t\|_{L^2(\mathbb R^3_+)}^2\big)\big]
   +\sigma^3(t)\big(\|\nabla u_t\|_{L^2(\mathbb R^3_+)}^2+\|\Delta d_t\|_{L^2(\mathbb R^3_+)}^2\big)
   \\&&\leq
   3\sigma^2(t)\sigma^\prime(t)\big(\|u_t\|_{L^2(\mathbb R^3_+)}^2+\|\nabla d_t\|_{L^2(\mathbb R^3_+)}^2\big)+C\sigma^3(t)\|\nabla d\|_{H^1(\mathbb R^3_+)}^4\|\nabla d_t\|_{L^2(\mathbb R^3_+)}^2.
\end{eqnarray}
Integrating the above inequality in $[0,t]$ and using (\ref{basci energy}) and (\ref{3190}), we get
\begin{eqnarray}
&&\big(\|u_t(\tau)\|_{L^2(\mathbb R^3_+)}^2+\|\nabla d_t(\tau)\|_{L^2(\mathbb R^3_+)}^2\big)+\int_\tau^t\big(\|\nabla u_t(s)\|_{L^2(\mathbb R^3_+)}^2+\|\Delta d_t(s)\|_{L^2(\mathbb R^3_+)}^2\big){\rm d}s\nonumber\\
&& \leq C_\tau.
\end{eqnarray}
This, combined with  the standard elliptic estimates,
implies \eqref{3180}.
\end{proof}

\subsection{Uniqueness}
To show the uniqueness of global strong solutions obtained in
Theorem \ref{th:E_0}, let $(u_i,d_i,p_i)$ be two strong solutions
to (\ref{system})--(\ref{i-b condition}) as in Theorem
\ref{th:E_0}. Set
$$\widetilde{u}=u_1-u_2, \ \widetilde{p}=p_1-p_2,\
\widetilde{d}=d_1-d_2.$$
Then we have
\begin{eqnarray} \label{system-i}
\begin{cases}
     \widetilde{u}_t+\widetilde{u}\cdot\nabla u_1
     +u_2\cdot\nabla\widetilde{u}+\nabla \widetilde{p}
     =\Delta{\widetilde{u}}-\nabla\cdot(\nabla \widetilde{d}\odot\nabla d_1)-\nabla\cdot(\nabla d_2\odot\nabla \widetilde{d}),\\
     \nabla\cdot \widetilde{u}=0,\\
     \widetilde{d}_t+\widetilde{u}\cdot\nabla d_1+u_2\cdot\nabla\widetilde{d}=\Delta{\widetilde{d}}+\langle\nabla\widetilde{d}, (\nabla d_1+\nabla d_2)\rangle d_1+|\nabla d_2|^2\widetilde{d},%\\[2mm]
\end{cases}
\end{eqnarray}
along with the initial and boundary conditions
\begin{eqnarray} \label{i-b condition-i}
\begin{cases}
     \widetilde{u}=\frac{\partial \widetilde{d}}{\partial x_3}=0,\ {\rm on}\
     \partial\mathbb R^3_+\times (0,\infty),\\
     (\widetilde{u},\widetilde{d})\rightarrow 0, \ {\rm{as}}\ |x|\rightarrow \infty, \\
     (\widetilde{u}, \ \widetilde{d})\big|_{t=0}=0,\ {\rm in}\ \mathbb{R}^3_+.
\end{cases}
\end{eqnarray}
Multiplying (\ref{system-i})$_1$ by $\widetilde{u}$ and
(\ref{system-i})$_3$ by $\Delta\widetilde{d}$, and applying integration by parts, (\ref{system-i})$_2$, (\ref{i-b condition-i}), the H\"older inequality, the Sobolev inequality and the Cauchy inequality, we obtain
\begin{eqnarray}
    &&\nonumber
    \frac{{\rm d}}{{\rm d}t}\big(\|\widetilde{u}\|_{L^2(\mathbb R^3_+)}^2
    +\|\nabla\widetilde{d}\|_{L^2(\mathbb R^3_+)}^2\big)
    +\big(1-\widetilde{G}(t)\big)\big(\|\nabla\widetilde{u}\|_{L^2(\mathbb R^3_+)}^2+\|\nabla^2\widetilde{d}\|_{L^2(\mathbb R^3_+)}^2\big)
    \\&\leq& \nonumber
    -\int_{\mathbb{R}^3_+}|\nabla d_2|^2\widetilde{d}\cdot\Delta\widetilde{d}{\rm d}x
    \\&\leq&\nonumber
    \int_{\mathbb{R}^3_+}|\nabla d_2|^2|\nabla\widetilde{d}|^2{\rm d}x+2\int_{\mathbb{R}^3_+}|\nabla d_2||\nabla^2 d_2||\widetilde{d}||\nabla\widetilde{d}|{\rm d}x
    \\&\leq&\nonumber
    \|\nabla d_2\|_{L^3(\mathbb R^3_+)}^2\|\nabla \widetilde{d}\|_{L^6(\mathbb R^3_+)}^2+2\int_{\mathbb{R}^3_+}|\nabla d_2||\nabla^2 d_2||\widetilde{d}||\nabla\widetilde{d}|{\rm d}x
    \\&\leq&\nonumber
    \|\nabla d_2\|_{L^3(\mathbb R^3_+)}^2\|\nabla^2 \widetilde{d}\|_{L^2(\mathbb R^3_+)}^2+\big\||\nabla d_2|^\frac12|\nabla^2 d_2|\big\|_{L^2(\mathbb R^3_+)}\big\|\nabla d_2\big\|_{L^3(\mathbb R^3_+)}^\frac12\|\widetilde{d}\|_{L^6(\mathbb R^3_+)}\|\nabla\widetilde{d}\|_{L^6(\mathbb R^3_+)}
    \\&\leq&\nonumber
    \|\nabla d_2\|_{L^3(\mathbb R^3_+)}^2\|\nabla^2 \widetilde{d}\|_{L^2(\mathbb R^3_+)}^2+C\|\nabla d_2\|_{L^3(\mathbb R^3_+)}\|\nabla^2\widetilde{d}\|_{L^2}^2\nonumber\\
    &&+C\big\||\nabla d_2|^\frac12|\nabla^2 d_2|\big\|_{L^2(\mathbb R^3_+)}^2\|\nabla \widetilde{d}\|_{L^2(\mathbb R^3_+)}^2,
\end{eqnarray}
where
$$\widetilde{G}(t)=\sum_{i=1}^2(\|u_i\|_{L^3(\mathbb R^3_+)}
+\|\nabla d_i\|_{L^3(\mathbb R^3_+)}+\|\nabla d_i\|_{L^3(\mathbb R^3_+)}^2).$$
By choosing a sufficiently small $\varepsilon_0$ and integrating over
$[0,t]$,  and applying Lemma \ref{le:L3},
we can conclude that
$$\|\widetilde{u}\|_{L^2(\mathbb R^3_+)}^2(t)+\|\nabla\widetilde{d}\|_{L^2(\mathbb R^3_+)}^2(t)\leq \|\widetilde{u}_0\|_{L^2(\mathbb R^3_+)}^2+\|\nabla\widetilde{d}_0\|_{L^2(\mathbb R^3_+)}^2=0$$
for any $t>0$. This implies that
$(u_1, d_1)\equiv (u_2,d_2)$ and  completes
the proof of uniqueness.

\subsection{Time-decay estimates}\label{sub3.3}

In this subsection, we will apply the continuity argument to
derive the time decay rates stated as in Theorem \ref{th:E_0}.

\begin{lemma} \label{le:3.2} There exists $\widetilde{C}>0$ such that if
\begin{equation}\label{apriori}
   \|\nabla d(t)\|_{L^2(\mathbb R^3_+)}\leq 2\widetilde{C}(1+t)^{-1},
\end{equation}
for any $t\in[0,T]$,
then
\begin{eqnarray}\label{nabla dL2}
   \|\nabla d(t)\|_{L^2(\mathbb R^3_+)}\leq \frac32\widetilde{C}(1+t)^{-1},
\end{eqnarray} for any $t\in[0,T]$.
Moreover, there exists a constant $C>0$ such that
\begin{eqnarray}\label{uL2}
   \|u(t)\|_{L^2(\mathbb R^3_+)}\leq C(1+t)^{-\frac34}
\end{eqnarray}  for any $t\in[0,T]$.
\end{lemma}
\begin{proof} Without loss of generality, we assume that $t\geq1$.
From (\ref{basci energy}), we have
\begin{eqnarray}
     \frac{d}{dt}\big(\|u\|_{L^2(\mathbb R^3_+)}^2+\|\nabla d\|_{L^2(\mathbb R^3_+)}^2\big)+\|\nabla u\|_{L^2(\mathbb R^3_+)}^2\leq 0.
\end{eqnarray} This, combined with (\ref{A1/2u}),
implies that
\begin{eqnarray}
   \label{energy-change}
   &&\frac{d}{dt}\big(\|u\|_{L^2(\mathbb R^3_+)}^2+\|\nabla d\|_{L^2(\mathbb R^3_+)}^2\big)+\rho\|u\|_{L^2(\mathbb R^3_+)}^2+\rho\|\nabla d\|_{L^2(\mathbb R^3_+)}^2\nonumber\\
   &&\leq
   C\rho\|{E}_\rho u\|_{L^2(\mathbb R^3_+)}^2+\rho\|\nabla d\|_{L^2(\mathbb R^3_+)}^2.
\end{eqnarray}
It follows from (\ref{u-halft})$_1$ that
\begin{eqnarray}\label{Erho}
    E_\rho u(t)&=& E_\rho e^{-\frac{t}{2}\mathbb{A}}u(\frac{t}2)-E_\rho\int_{\frac{t}{2}}^te^{-(t-s)\mathbb{A}}\mathbb{P}(u\cdot\nabla u)(s){\rm d}s
    \nonumber\\
    &&-E_\rho\int_{\frac{t}{2}}^te^{-(t-s)\mathbb{A}}\mathbb{P}\nabla\cdot(\nabla d\odot\nabla d)(s){\rm d}s\nonumber
    \\
    &=&
    I_1-I_2-I_3.
\end{eqnarray}
By (\ref{A}) and calculations similar to \cite{Bocher} (page 150), we have
\begin{eqnarray}
    \nonumber
    I_2&=&\int_{\frac{t}{2}}^t\big[\int_0^\rho e^{-\lambda(t-s)}{\rm d}(E_\lambda\left(\mathbb{P}\left(u\cdot\nabla u\right)\right)(s))\big]{\rm d}s
    \\&=&\nonumber
    \int_{\frac{t}{2}}^t e^{-\rho(t-s)}E_\rho\left(\mathbb{P}\left(u\cdot\nabla u\right)\right)(s){\rm d}s
    \\&&
    +\int_{\frac{t}{2}}^t(t-s)\big[\int_0^\rho e^{-\lambda(t-s)}E_\lambda\big(\mathbb{P}\big(u\cdot\nabla u\big)\big)(s){\rm d}\lambda\big]{\rm d}s.
\end{eqnarray}
This, combined with the classical estimate (see \cite{Bocher} Lemma
4.3):
\begin{eqnarray}
    \|E_\lambda\mathbb{P}(u\cdot\nabla)u\|_{L^2(\mathbb R^3_+)}\leq C\lambda^{\frac{5}{4}}\|u\|_{L^2(\mathbb R^3_+)}^2,
\end{eqnarray}
implies that
\begin{eqnarray}\label{I2}
   \big\|I_2\big\|_{L^2(\mathbb R^3_+)}\leq C\rho^{\frac54}\int_{\frac{t}{2}}^t\|u(s)\|_{L^2(\mathbb R^3_+)}^2{\rm d}s.
\end{eqnarray}

By the Minkowski inequality and the
H\"older inequality, Lemma
\ref{le:1} for $p=2$ and $q=r\in(1,2]$, the boundedness of
$$\mathbb{P}: L^r(\mathbb{R}^3_+,\mathbb R^3)\mapsto
L^r_\sigma(\mathbb{R}^3_+,\mathbb R^3),\ \forall\ r\in(1,\infty),$$
the fact
$$\|E_\rho f\|_{L^2(\mathbb R^3_+)}\leq \|f\|_{L^2(\mathbb R^3_+)}, \ f\in L^2_\sigma(\mathbb R^3_+,\mathbb R^3),$$
and (\ref{3.19}), we can estimate $I_3$ as follows.
\begin{eqnarray}
    \label{nonlinear term}
    \|-I_3\|_{L^2(\mathbb R^3_+)}
    &&=
    \big\|E_\rho\int_{\frac{t}{2}}^te^{-(t-s)\mathbb{A}}\mathbb{P}\nabla\cdot(\nabla d\odot\nabla d)(s){\rm d}s\big\|_{L^2(\mathbb R^3_+)}
    \nonumber\\
    &&\leq
    \big\|\int_{\frac{t}{2}}^te^{-(t-s)\mathbb{A}}\mathbb{P}\nabla\cdot(\nabla d\odot\nabla d)(s){\rm d}s\big\|_{L^2(\mathbb R^3_+)}
    \nonumber\\
    &&\leq
    \int_{\frac{t}{2}}^t\big\|e^{-(t-s)\mathbb{A}}\mathbb{P}\nabla\cdot(\nabla d\odot\nabla d)(s)\big\|_{L^2(\mathbb R^3_+)}{\rm d}s
    \nonumber\\
    &&\leq
    C\int_{\frac{t}{2}}^t(t-s)^{-\frac32\big(\frac{1}{r}-\frac12\big)}\big\|\mathbb{P}\nabla\cdot(\nabla d\odot\nabla d)(s)\big\|_{L^r(\mathbb R^3_+)}{\rm d}s
    \nonumber\\
    &&\leq
    C\int_{\frac{t}{2}}^t(t-s)^{-\frac32\big(\frac{1}{r}-\frac12\big)}\|\nabla d(s)\|_{L^{\frac{2r}{2-r}}(\mathbb R^3_+)}\|\nabla^2 d(s)\|_{L^2(\mathbb R^3_+)}{\rm d}s
    \nonumber\\
    &&\leq
    C\int_{\frac{t}{2}}^t(t-s)^{-\frac32\big(\frac{1}{r}-\frac12\big)}
    \|\nabla d(s)\|_{L^2(\mathbb R^3_+)}^{\frac{3-2r}{r}}\|\nabla^2 d(s)\|_{L^2(\mathbb R^3_+)}^{1+\frac{3r-3}{r}}{\rm d}s.
\end{eqnarray}
This, combined with the H\"older inequality, (\ref{apriori}), and Lemma \ref{key-estimate}, yields
\begin{eqnarray}
    \label{nonlinear term-2}
    &&\|I_3\|_{L^2(\mathbb R^3_+)}
    \nonumber\\
    &&\leq
    C\big(\int_{\frac{t}{2}}^t(t-s)^{-2(\frac{1}{r}-\frac12)}\|\nabla d(s)\|_{L^2(\mathbb R^3_+)}^{\frac{4(3-2r)}{3r}}{\rm d}s\big)^{\frac34}
    \big(\int_{\frac{t}{2}}^t\|\nabla^2 d(s)\|_{L^2(\mathbb R^3_+)}^{4(1+\frac{3r-3}{r})}{\rm d}s\big)^{\frac14}
    \nonumber\\
    &&\leq
    C\big(\int_{\frac{t}{2}}^t(t-s)^{-2(\frac{1}{r}-\frac12)}(1+s)^{-\frac{4(3-2r)}{3r}}{\rm d}s\big)^{\frac{3}{4}}\nonumber\\
    &&\ \ \cdot\big(\sup\limits_{s\in [\frac{t}2,t]}\|\nabla^2 d(s)\|_{L^2(\mathbb R^3_+)}\big)
    ^{\frac72-\frac{3}{r}}\big(\int_{\frac{t}{2}}^t\|\nabla^2 d(s)\|_{L^2(\mathbb R^3_+)}^2{\rm d}s\big)^{\frac14}
    \nonumber\\
    &&\leq
    C\big(\int_{\frac{t}{2}}^t(t-s)^{-2(\frac{1}{r}-\frac12)}(1+s)^{-\frac{4(3-2r)}{3r}}{\rm d}s\big)^{\frac{3}{4}}\nonumber\\
    &&\leq
    C\big(1+t\big)^{\frac34(1-2(\frac{1}{r}-\frac12)-\frac{4(3-2r)}{3r})}
    \leq
    C(1+t)^{-\frac34},
\end{eqnarray}
for some $r\in (1,\frac{18}{17})$.

To estimate $I_1$, we need to estimate
the upper bound of
$\|u(t)\|_{L^a(\mathbb R^3_+)}$ for $a\in(1,\frac32)$.
From (\ref{u})$_1$, Lemma \ref{le:1} and Lemma \ref{le:B}, we have that
\begin{eqnarray} \label{u-La}
\|u(t)\|_{L^a(\mathbb R^3_+)}
&\leq&Ct^{-\frac32(1-\frac{1}{a})}\|u_0\|_{L^1(\mathbb R^3_+)}
+\big\|\int_0^te^{-(t-s)\mathbb{A}}\mathbb{P}\big(u\cdot\nabla
u-\nabla\cdot(\nabla d\odot\nabla d)\big){\rm
d}s\big\|_{L^a(\mathbb R^3_+)}\nonumber\\
&\leq&
Ct^{-\frac32(1-\frac{1}{a})}\|u_0\|_{L^1(\mathbb R^3_+)}
+\int_0^t(t-s)^{-\frac12-\frac32(1-\frac{1}{a})}{\rm
d}s\leq Ct^{\frac12-\frac32(1-\frac{1}{a})},
\end{eqnarray}
where we have used the following estimate: for any $a\in (1,\frac32)$ and
$t>0$,
\begin{eqnarray}\label{3.49}
&&\big\|\int_0^te^{-(t-s)\mathbb{A}}\mathbb{P}\left(u\cdot\nabla u-\nabla\cdot(\nabla d\odot\nabla d)\right){\rm d}s\big\|_{L^a(\mathbb R^3_+)}
     \nonumber\\
     &=&\sup_{\varphi\in
C^\infty_{0,\sigma}(\mathbb{R}^3_+,\mathbb R^3), \|\varphi\|_{L^{\frac{a}{a-1}}(\mathbb R^3_+)}\le 1}\big|\langle\int_0^te^{-(t-s)\mathbb{A}}\mathbb{P}\big(u\cdot\nabla u-\nabla\cdot(\nabla d\odot\nabla d)\big){\rm d}s,\varphi\rangle\big|
     \nonumber\\
     &=&\nonumber
     \sup_{\varphi\in
C^\infty_{0,\sigma}(\mathbb{R}^3_+,\mathbb R^3), \|\varphi\|_{L^{\frac{a}{a-1}}(\mathbb R^3_+)}\le 1}\big|\langle\int_0^te^{-(t-s)\mathbb{A}}\mathbb{P}\nabla\cdot\big(u\otimes u-\nabla d\odot\nabla d\big){\rm d}s,\varphi\rangle\big|
     \\&=&\nonumber
    \sup_{\varphi\in
C^\infty_{0,\sigma}(\mathbb{R}^3_+,\mathbb R^3), \|\varphi\|_{L^{\frac{a}{a-1}}(\mathbb R^3_+)}\le 1} \big|\langle\int_0^t\big(u\otimes u-\nabla d\odot\nabla d\big){\rm d}s,
     \nabla e^{-(t-s)\mathbb{A}}\varphi\rangle\big|
     \\
     &\leq&\nonumber
     \sup_{\varphi\in
C^\infty_{0,\sigma}(\mathbb{R}^3_+,\mathbb R^3), \|\varphi\|_{L^{\frac{a}{a-1}}(\mathbb R^3_+)}\le 1}\int_0^t\|\nabla e^{-(t-s)\mathbb{A}}\varphi\|_{L^\infty(\mathbb R^3_+)}\|u\otimes u-\nabla d\odot\nabla d\|_{L^1(\mathbb R^3_+)}{\rm d}s
     \\&\leq&
     \int_0^t(t-s)^{-\frac12-\frac32(1-\frac{1}{a})}
     \big(\|u(s)\|_{L^2(\mathbb R^3_+)}^2+\|\nabla d(s)\|_{L^2(\mathbb R^3_+)}^2\big){\rm d}s.
\end{eqnarray}

Now we can estimate $\|I_1\|_{L^2(\mathbb R^3_+)}$ by
\begin{eqnarray} \label{energy-change-2-1}
   \|I_1\|_{L^2(\mathbb R^3_+)}
   &\leq&
   C\|e^{-\frac{t}{2}\mathbb{A}}u(\frac{t}2)\|_{L^2(\mathbb R^3_+)}
   \nonumber\\
   &\leq&
   Ct^{-\frac{3}{2}(\frac{1}{a}-\frac{1}{2})}\|u(\frac{t}2)\|_{L^a(\mathbb R^3_+)}
   \nonumber\\
   &\leq&
   Ct^{-\frac{3}{2}(\frac{1}{a}-\frac{1}{2})}t^{\frac{1}{2}-\frac{3}{2}
   (1-\frac{1}{a})}
   \nonumber\\
   &\leq& C(1+t)^{-\frac14}.
\end{eqnarray}

Putting (\ref{energy-change}), (\ref{I2}),
(\ref{nonlinear term-2}), (\ref{energy-change-2-1}), together with
(\ref{apriori})  yields that
%\begin{eqnarray} \label{energy-change-2}\begin{split}
%   &\frac{d}{dt}\left(\|u\|_{L^2}^2+\|\nabla d\|_{L^2}^2\right)+\rho\|u\|_{L^2}^2+\rho\|\nabla d\|_{L^2}^2
%   \leq
%   C\rho\|\mathbb{E}_\rho u\|_{L^2}^2+\rho\|\nabla d\|_{L^2}^2
 %  \\&\leq
%   C\rho\left((1+t)^{-3\left(\frac{1}{a}-\frac{1}{2}\right)}+\rho^{\frac52}\left(\int_\tau^t\|u\|_{L^2}^2(s){\rm d}s\right)^2+(1+t)^{6-\frac{15}{2r}}\right).
%\end{split}\end{eqnarray}
%Now let
%By using (\ref{energy-change-2}), we have
\begin{eqnarray}\label{diff_ineq}
    \label{g}
    \frac{dg}{dt}+\rho g\leq C\rho\big[(1+t)^{-\frac12}
    +\rho^{\frac52}\big(\int_{\frac{t}{2}}^t\|u(s)\|_{L^2(\mathbb R^3_+)}^2{\rm
    d}s\big)^2+(1+t)^{-\frac32}\big],
\end{eqnarray}
where
$$g(t)=\|u(t)\|_{L^2(\mathbb R^3_+)}^2+\|\nabla d(t)\|_{L^2(\mathbb R^3_+)}^2.$$
For $k$ sufficiently large, let
$\rho=\frac{k}{1+t}$  and multiply \eqref{diff_ineq} by $(1+t)^k$.
Then we have
\begin{eqnarray}\label{3.47}
   \nonumber
   \big((1+t)^kg\big)^\prime&\leq& C(1+t)^{-1+k}\big((1+t)^{-\frac12}+(1+t)^{-\frac32}\big)
   \\&\leq&
   C(1+t)^{-\frac32+k}.
\end{eqnarray}
This, after integrating over $[1,t]$ and applying Lemma \ref{le:B},
implies that
\begin{equation}
\label{ul2decay1}\big(\|u(t)\|_{L^2(\mathbb R^3_+)}+\|\nabla d(t)\|_{L^2(\mathbb R^3_+)}\big)\leq C(1+t)^{-\frac14}.\end{equation} Inserting
(\ref{ul2decay1}) into (\ref{3.49}) first  and then (\ref{u-La}), we obtain that
\begin{equation}
    \|u(t)\|_{L^a(\mathbb R^3_+)}\leq Ct^{-\frac32(1-\frac{1}{a})}.
\end{equation}
This can be used to improve estimate of $I_1$ to
$$\|I_1(t)\|_{L^2(\mathbb R^3_+)}\leq C(1+t)^{-\frac34},$$
which can then be used to improve (\ref{3.47}) to
\begin{eqnarray}\label{uLr}
   \big((1+t)^kg\big)^\prime&\leq&
   C(1+t)^{-1+k}(1+t)^{-\frac32}.
\end{eqnarray}
Thus we obtain
\begin{eqnarray} \label{3/8}
   \|u(t)\|_{L^2(\mathbb R^3_+)}\leq C(1+t)^{-\frac34}
\end{eqnarray}
so that (\ref{uL2}) holds.

To show (\ref{nabla dL2}), first observe that by (\ref{u})$_2$,
(\ref{apriori}), (\ref{3/8}) and Lemma \ref{le:1}, we have that
for small $\delta>0$,
\begin{eqnarray} \label{d-w0-Lb}
    \nonumber
    &&\|(d-e_3)(t)\|_{L^{\frac{6+\delta}{5}}(\mathbb R^3_+)}
    \\&\leq&\nonumber
    C_{0,\frac{6+\delta}{5},1,3}t^{-\frac32\left(1-\frac{5}{6+\delta}\right)}\|d_0-e_3\|_{L^1(\mathbb R^3_+)}
    \\&+&\nonumber
    \mathcal{C}\int_0^t(t-s)^{-\frac32(\frac56-\frac{5}{6+\delta})}
    \big(\|u\cdot\nabla d\|_{L^{\frac65}(\mathbb R^3_+)}+\||\nabla d|^2\|_{L^{\frac65}(\mathbb R^3_+)}\big){\rm d}s
    \\&\leq& \nonumber
    \mathcal{C}t^{-\frac32\left(1-\frac{5}{6+\delta}\right)}\nonumber\\
    &+&\mathcal{C}\int_0^t(t-s)^{-\frac32(\frac56-\frac{5}{6+\delta})}\big(\|u\|_{L^3(\mathbb R^3_+)}\|\nabla d\|_{L^2(\mathbb R^3_+)}+\|\nabla d\|_{L^3(\mathbb R^3_+)}\|\nabla d\|_{L^2(\mathbb R^3_+)}\big){\rm d}s\nonumber
    \\&\leq& \nonumber
    \mathcal{C}t^{-\frac32(1-\frac{5}{6+\delta})}+\mathcal{C}\widetilde{C}\varepsilon_0\int_0^t(t-s)^{-\frac32(\frac56-\frac{5}{6+\delta})}(1+s)^{-1}{\rm d}s
    \\&\leq&
    \mathcal{C}t^{-\frac32(\frac56-\frac{5}{6+\delta})},
\end{eqnarray}
provided $\varepsilon_0\leq \widetilde{C}^{-1}$.
From here to the end of this section, $C$ denotes a positive constant depending on $\widetilde{C}$,
and $\mathcal{C}$ denotes a positive constant
that is independent of $\widetilde{C}$.

It follows from (\ref{u})$_2$ that
\begin{eqnarray}
     \nabla d(t)=\nabla e^{\frac{t}{2}\Delta}(d-w_0)(\frac{t}2)-\int_{\frac{t}{2}}^t\nabla e^{(t-s)\Delta}\left(u(s)\cdot\nabla d(s)-|\nabla d(s)|^2d(s)\right){\rm d}s.
\end{eqnarray}
From Lemma \ref{le:1} and (\ref{d-w0-Lb}), we obtain that for any
small $\delta>0$,
\begin{eqnarray}\label{3.59}
     \nonumber
     \|\nabla d(t)\|_{L^2(\mathbb R^3_+)}&\leq& C_{1,2,\frac{6+\delta}5,3}t^{-\frac12-\frac32(\frac{5}{6+\delta}-\frac{1}{2})}\|(d-w_0)(\frac{t}2)
     \|_{L^{\frac{6+\delta}{5}}(\mathbb R^3_+)}
     \\&+&\nonumber
     C\int_{\frac{t}{2}}^t(t-s)^{-\frac12-\frac32(\frac{1}{q}-\frac12)}\big(\|(u\cdot\nabla d)(s)\|_{L^q(\mathbb R^3_+)}
     +\||\nabla d(s)|^2\|_{L^q(\mathbb R^3_+)}\big){\rm d}s
     \\&\leq&
     \mathcal{C}t^{-1}+C\int_{\frac{t}{2}}^t(t-s)^{\frac14-\frac{3}{2q}}\big(\|(u\cdot\nabla d)(s)\|_{L^q(\mathbb R^3_+)}
     +\|\nabla d(s)\|_{L^{2q}(\mathbb R^3_+)}^2\big){\rm d}s,
\end{eqnarray}
where we choose $q\in (\frac65,2)$ so that $\frac14-\frac{3}{2q}>-1$ and
$\displaystyle\lim_{q\rightarrow \frac65}(\frac14-\frac{3}{2q})= -1$.

It is readily seen that if we choose the constant $\widetilde{C}=2\mathcal{C}$ in (\ref{apriori}),
then it follows from \eqref{apriori} and \eqref{3.59} that
\begin{eqnarray} \label{3.12}
    \nonumber
    \|\nabla d(t)\|_{L^2(\mathbb R^3_+)}&\leq& \widetilde{C}(1+t)^{-1}\\
     &+&C\int_{\frac{t}{2}}^t(t-s)^{\frac14-\frac{3}{2q}}\big(\|u(s)\|_{L^2(\mathbb R^3_+)}\|\nabla d(s)\|_{L^{\frac{2q}{2-q}}(\mathbb R^3_+)}+\|\nabla d(s)\|_{L^{2q}(\mathbb R^3_+)}^2\big){\rm d}s.
\end{eqnarray}
Now we need to recall several interpolation inequalities and Sobolev's inequalities:
\begin{eqnarray}
    \label{3.13}\|\nabla d(s)\|_{L^{\frac{2q}{2-q}}(\mathbb R^3_+)}
    &\leq& C\|\nabla d(s)\|_{L^2(\mathbb R^3_+)}^{\frac{3}{q}-2}\|\nabla d(s)\|_{L^6(\mathbb R^3_+)}^{3-\frac{3}{q}}\nonumber\\
    &\leq& C\|\nabla d(s)\|_{L^2(\mathbb R^3_+)}^{\frac{3}{q}-2}\|\nabla^2 d(s)\|_{L^2(\mathbb R^3_+)}^{3-\frac{3}{q}},\\
    \label{3.14}\|\nabla d(s)\|_{L^{2q}(\mathbb R^3_+)}&\leq& C\|\nabla d(s)\|_{L^2(\mathbb R^3_+)}^{\frac{3-q}{2q}}
    \|\nabla^2 d(s)\|_{L^2(\mathbb R^3_+)}^{\frac{3q-3}{2q}},\\
    \label{3.15}
    \|\nabla d(s)\|_{L^{\frac{2q}{2-q}}(\mathbb R^3_+)}&\leq& C\|\nabla d(s)\|_{L^3(\mathbb R^3_+)}^{\frac{6}{q}-4}
    \|\nabla^2d(s)\|_{L^2(\mathbb R^3_+)}^{5-\frac{6}{q}},\\
    \label{3.16}
    \|\nabla d(s)\|_{L^{2q}(\mathbb R^3_+)}&\leq& C\|\nabla d(s)\|_{L^2(\mathbb R^3_+)}^{\frac{3}{q}-2}\|\nabla d(s)\|_{L^3(\mathbb R^3_+)}^{3-\frac{3}{q}},
\end{eqnarray}
where $q\in (\frac65,\frac32)$.

Applying \eqref{3.13}, \eqref{3.14}, \eqref{3.15}, \eqref{3.16}, \eqref{apriori}£¬ \eqref{3/8} and Lemma \ref{le:3.7}, we can obtain
\begin{eqnarray}\label{3.65}
    &&\|u(s)\|_{L^2(\mathbb R^3_+)}\|\nabla d(s)\|_{L^{\frac{2q}{2-q}}(\mathbb R^3_+)}\nonumber\\
    &&=\|u(s)\|_{L^2(\mathbb R^3_+)}\|\nabla d(s)\|_{L^{\frac{2q}{2-q}}(\mathbb R^3_+)}^{\frac56}
    \|\nabla d(s)\|_{L^{\frac{2q}{2-q}}(\mathbb R^3_+)}^{\frac16}
    \nonumber\\
    &&\leq\nonumber
    C\|u(s)\|_{L^2(\mathbb R^3_+)}\|\nabla d(s)\|_{L^2(\mathbb R^3_+)}^{\frac{5}{2q}-\frac53}\|\nabla^2 d(s)\|_{L^2(\mathbb R^3_+)}^{\frac{10}{3}-\frac{7}{2q}}
    \|\nabla d(s)\|_{L^3(\mathbb R^3_+)}^{\frac{1}{q}-\frac23}
    \\
    &&\leq
    C\varepsilon_0^{\frac{1}{q}-\frac23}(1+s)^{\frac{11}{12}-\frac{5}{2q}},
\end{eqnarray}
and
\begin{eqnarray}\label{3.66}
    \nonumber
    \|\nabla d(s)\|_{L^{2q}(\mathbb R^3_+)}^2
    &=&\|\nabla d(s)\|_{L^{2q}(\mathbb R^3_+)}^{\frac32}\|\nabla d(s)\|_{L^{2q}(\mathbb R^3_+)}^{\frac12}
    \\&\leq&\nonumber
    C\|\nabla d(s)\|_{L^2(\mathbb R^3_+)}^{\frac{15}{4q}-\frac74}
    \|\nabla^2 d(s)\|_{L^2(\mathbb R^3_+)}^{\frac{9q-9}{4q}}
    \|\nabla d(s)\|_{L^3(\mathbb R^3_+)}^{\frac32-\frac{3}{2q}}
    \\&\leq&
    C\varepsilon_0^{\frac32-\frac3{2q}}(1+s)^{\frac{7}{4}-\frac{15}{4q}} {\rm{d}}s
\end{eqnarray}
for $s\in[\frac{t}2,t]$.

Substituting \eqref{3.65} and \eqref{3.66} into \eqref{3.12} and applying Lemma 2.2 yields that
\begin{eqnarray}
    \nonumber\label{d-1}
    &&\|\nabla d(t)\|_{L^2(\mathbb R^3_+)}
     \\&&\leq\nonumber
     \widetilde{C}(1+t)^{-1}
     +C\varepsilon_0^{\frac{1}{q}-\frac23}\int_{\frac{t}{2}}^t(t-s)^{\frac14-\frac{3}{2q}}(1+s)^{\frac{11}{12}-\frac{5}{2q}}
    {\rm d}s
    \\
    &&\quad+ C\varepsilon_0^{\frac32-\frac3{2q}}\int_{\frac{t}{2}}^t(t-s)^{\frac14-\frac{3}{2q}}(1+s)^{\frac{7}{4}-\frac{15}{4q}} {\rm{d}}s
    \nonumber\\
    &&\leq\nonumber
   \widetilde{C}(1+t)^{-1}+C\varepsilon_0^{\frac1q-\frac23}(1+t)^{\frac{13}6-\frac4q}+C\varepsilon_0^{\frac32-\frac3{2q}}
   (1+t)^{3-\frac{21}{4q}}
    \\
    &&\leq\frac{3\widetilde{C}}{2}(1+t)^{-1}
\end{eqnarray}
provided that we choose $q\in (\frac65, \frac{24}{19})$
and $\varepsilon_0$ so small that
$$C\varepsilon_0^{\frac1q-\frac23}+ C\varepsilon_0^{\frac32-\frac3{2q}}\le \frac{\widetilde{C}}2.$$
This implies \eqref{nabla dL2} and completes the proof.
\end{proof}

By the standard continuity argument, we can then complete the proof of  Lemma 3.5.

\begin{corollary} \label{cor-1} Under the same assumptions of Theorem 1.1, if
$(u,d)$ is the global strong solution obtained by Theorem 1.1, then
\begin{eqnarray}
    \|u(t)\|_{L^2(\mathbb R^3_+)}\leq C(1+t)^{-\frac34},
    \ \ \|\nabla d(t)\|_{L^2(\mathbb R^3_+)}\leq C(1+t)^{-1}, \ \forall \ t>0.
\end{eqnarray}
\end{corollary}

With some further calculations, we also have
\begin{corollary} \label{cor-2} Under the same assumptions of Theorem 1.1, if
$(u,d)$ is the global strong solution obtained by Theorem 1.1, then  it holds
that
\begin{eqnarray}
    &&\label{d-w}\|(d-e_3)(t)\|_{L^r(\mathbb R^3_+)}\leq
    Ct^{-\frac32(1-\frac{1}{r})},\ \forall t>0,\ \forall r\in (1, 2], \\
    && \|\nabla d(t)\|_{L^2(\mathbb R^3_+)}\leq C(1+t)^{-\frac54}, \ \forall t>0.
\end{eqnarray}
\end{corollary}

\begin{proof} By  (\ref{u})$_2$ and Corollary \ref{cor-1}, we have that for any $r\in (1, 2]$, it holds
\begin{eqnarray} \label{d-w0-Lr}
    \nonumber
    &&\|(d-e_3)(t)\|_{L^r(\mathbb R^3_+)}\\
    &\leq& Ct^{-\frac32(1-\frac{1}{r})}\|d_0-e_3\|_{L^1(\mathbb R^3_+)}\nonumber
    \\&+&\nonumber
    C\int_0^t(t-s)^{-\frac32(1-\frac{1}{r})}\big(\|u(s)\|_{L^2(\mathbb R^3_+)}\|\nabla d(s)\|_{L^2(\mathbb R^3_+)}
    +\|\nabla d(s)\|_{L^2(\mathbb R^3_+)}^2\big){\rm d}s
    \\&\leq& \nonumber
    Ct^{-\frac32(1-\frac{1}{r})}+C\int_0^t(t-s)^{-\frac32(1-\frac{1}{r})}\big[(1+s)^{-\frac74}+(1+s)^{-2}\big]{\rm d}s
    \\&\leq&
    Ct^{-\frac32(1-\frac{1}{r})}.
\end{eqnarray}
By (\ref{u-halft})$_2$, Lemma \ref{le:1}, (\ref{d-w0-Lr}), (\ref{3.13}), and (\ref{3.14}), we then have
\begin{eqnarray} \label{nable-d-L2}
    \nonumber
    &&\|\nabla d(t)\|_{L^2(\mathbb R^3_+)}\leq Ct^{-\frac12}\|(d-e_3)(\frac{t}2)\|_{L^2(\mathbb R^3_+)}
    \\
    &&\nonumber
    \ \ +\int_{\frac{t}{2}}^t(t-s)^{\frac14-\frac{3}{2q}}\big(\|u(s)\|_{L^2(\mathbb R^3_+)}\|\nabla d(s)\|_{L^2(\mathbb R^3_+)}^{\frac{3}{q}-2}+\|\nabla d(s)\|_{L^2(\mathbb R^3_+)}^{\frac{3-q}{q}}\big)\|\nabla d(s)\|_{H^1(\mathbb R^3_+)}^{\frac{3q-3}{q}}{\rm d}s
    \\
    &&\leq\nonumber
    C(1+t)^{-\frac54}
    +\int_{\frac{t}{2}}^t(t-s)^{\frac14-\frac{3}{2q}}\big[(1+s)^{\frac54-\frac{3}{q}}+(1+s)^{1-\frac{3}{q}}\big]{\rm d}s
    \\
    &&\leq
    C(1+t)^{-\frac54}
    +C(1+t)^{\frac52-\frac{9}{2q}}
\end{eqnarray}
for any $q\in(\frac65,\frac32)$.

Since $\displaystyle \lim_{q\rightarrow\frac65} (\frac52-\frac{9}{2q})= -\frac54$,
it follows from \eqref{nable-d-L2} that  for any $\epsilon>0$, there exists $C_\epsilon>0$ such that
\begin{eqnarray}\label{3.73}
\|\nabla d(t)\|_{L^2(\mathbb R^3_+)}\leq  C_\epsilon(1+t)^{-\frac54+\epsilon}.
\end{eqnarray}
Substituting \eqref{3.73} into \eqref{nable-d-L2} and running the same argument as in \eqref{nable-d-L2}, we would obtain
the sharp estimate
\begin{eqnarray}\label{3.74}
\|\nabla d(t)\|_{L^2(\mathbb R^3_+)}\leq  C(1+t)^{-\frac54}.
\end{eqnarray}
This complets the proof.
\end{proof}

\begin{corollary} \label{le:3} Under the same assumptions of Theorem 1.1, if
$(u,d)$ is the global strong solution obtained by Theorem 1.1, then
the following estimates hold
\begin{eqnarray}
    &&\|u(t)\|_{L^{r}(\mathbb R^3_+)}\leq Ct^{-\frac32(1-\frac{1}{r})},\\
    &&\|\nabla d(t)\|_{L^s(\mathbb R^3_+)}\leq Ct^{-\frac12-\frac32(1-\frac{1}{s})}
\end{eqnarray}
for any $r\in(1,2]$ and $s\in[1,2]$.
\end{corollary}

\begin{proof} By (\ref{3.49}), Corollary \ref{cor-1}, and Corollary \ref{cor-2}, we have that
for any $\tilde{r}\in(1,\frac32)$, it holds
\begin{eqnarray}
     \nonumber
     \|u(t)\|_{L^{\tilde{r}}(\mathbb R^3_+)}&\leq& Ct^{-\frac32\left(1-\frac{1}{\tilde{r}}\right)}\|u_0\|_{L^1(\mathbb R^3_+)} +C\int_0^t(t-s)^{-\frac12-\frac32\left(1-\frac{1}{\tilde{r}}\right)}(1+s)^{-\frac32}{\rm d}s
     \\&\leq&
     Ct^{-\frac32\left(1-\frac{1}{\tilde{r}}\right)}.
\end{eqnarray} Then for any $r\in [\frac32,2)$, using the interpolation inequality, we have
\begin{eqnarray}
    \label{u-r-bc}
    \nonumber
    \|u(t)\|_{L^r(\mathbb R^3_+)}&\leq& \|u(t)\|_{L^{\tilde{r}}(\mathbb R^3_+)}^\alpha\|u(t)\|_{L^{2}(\mathbb R^3_+)}^{1-\alpha}\
    \\&\leq&
    Ct^{-\frac32(1-\frac{1}{\tilde{r}})\alpha}t^{-\frac34(1-\alpha)}
    \leq
    Ct^{-\frac32\left(1-\frac{1}{r}\right)},
\end{eqnarray}
where $\alpha\in (0,1)$ satisfies $\frac{\alpha}{\tilde{r}}+\frac{1-\alpha}{2}=\frac{1}{r}$.

While, by (\ref{u})$_2$,  (\ref{p-q-2}), Lemma \ref{key-estimate}, Corollary \ref{cor-1}
and Corollary \ref{cor-2}, we have that
\begin{eqnarray}
    \nonumber\label{nabel d-L1}
    \|\nabla d(t)\|_{L^1(\mathbb R^3_+)}&\leq&
    Ct^{-\frac12}\|d_0-e_3\|_{L^1(\mathbb R^3_+)}\\
    &+& C\int_0^t(t-s)^{-\frac12}\big(\|u(s)\|_{L^2(\mathbb R^3_+)}
    \|\nabla d(s)\|_{L^2(\mathbb R^3_+)}+\|\nabla d(s)\|_{L^2(\mathbb R^3_+)}^2\big){\rm d}s\nonumber
    \\&\leq& \nonumber
    Ct^{-\frac12}+C\int_0^t(t-s)^{-\frac12}\left[(1+s)^{-2}+(1+s)^{-\frac52}\right]{\rm d}s
    \\&\leq&
    Ct^{-\frac12}.
\end{eqnarray}
This, combined with \eqref{3.74} and the interpolation inequality, implies that for any $1\le s\le 2$,
\begin{eqnarray}
   \|\nabla d(t)\|_{L^s(\mathbb R^3_+)}\leq \|\nabla d(t)\|_{L^1(\mathbb R^3_+)}^{\frac{2}{s}-1}
   \|\nabla d(t)\|_{L^2(\mathbb R^3_+)}^{2-\frac{2}{s}}\leq Ct^{-\frac12-\frac32(1-\frac{1}{s})}.
\end{eqnarray}
This completes the proof.
\end{proof}

\begin{lemma} \label{le:4} Assume there exist $C_1, C_2>0$ such that
if the global strong solution $(u,d)$ given by Theorem 1.1 satisfies
\begin{eqnarray}
   &&\|\nabla u(t)\|_{L^6(\mathbb R^3_+)}
   \leq 2C_1t^{-\frac74}, \label{nabel-u-assumption}\\
   &&\|\nabla^2 d(t)\|_{L^6(\mathbb R^3_+)}\leq 2C_2t^{-\frac94}
\end{eqnarray} for any $t\in(0,T]$.
Then the following estimates
\begin{eqnarray}
   &&\|u(t)\|_{L^r(\mathbb R^3_+)}\leq Ct^{-\frac32(1-\frac{1}{r})},\label{u-a} \\
   &&\label{d-a}\|(d-e_3)(t)\|_{L^r(\mathbb R^3_+)}
   \leq Ct^{-\frac32(1-\frac{1}{r})},\\
   &&\|\nabla d(t)\|_{L^r(\mathbb R^3_+)}
   \leq Ct^{-\frac12-\frac32(1-\frac{1}{r})},\label{nabel-d-a}\\
   &&\label{c4}\|\nabla u(t)\|_{L^6(\mathbb R^3_+)}
   \leq \frac32C_1t^{-\frac74},\\
   &&\label{c5}\|\nabla^2 d(t)\|_{L^6(\mathbb R^3_+)}\leq \frac32C_2t^{-\frac94}
\end{eqnarray}
hold for any $t\in(0,T]$.
Here $r\in [2,\infty]$.
\end{lemma}
\begin{proof} Without loss of generality, we may assume that $t\geq1$.
We divide the proof into two steps:

%In fact, if $t<1$, $\sigma(t)=t$, then (\ref{3.19}) implies that $t\|\nabla u\|_{L^2}^2\leq C$, which yields $t^{\frac14}\cdot t\|\nabla u\|_{L^2}^2\leq C$, then we directly get $\|\nabla u\|_{L^2}^2\leq Ct^{-\frac54}$. Now we begin the proof.

\noindent{\it Step I}. By Gagliardo-Nirenberg inequality \footnote{Indeed, to see this, one has only to extend $u$ from $\mathbb{R}^3_+$ to $\mathbb{R}^3$ by odd or even extension and then apply the $\mathbb{R}^3$-version.} and Corollary \ref{le:3}, we know that for any $0<t\leq T$,
\begin{eqnarray}
    \label{u-r-bc}
    \nonumber
    \|u(t)\|_{L^\infty(\mathbb R^3_+)}&\leq& \|u(t)\|_{L^2(\mathbb R^3_+)}^{\frac14}\|\nabla u(t)\|_{L^{6}(\mathbb R^3_+)}^{\frac34}
    \\&\leq&
    Ct^{-\frac34\cdot\frac14-\frac74\cdot\frac34}
    \leq
    Ct^{-\frac32}.
\end{eqnarray}
Then the case for $r\in (2,\infty)$ directly follows from the interpolation inequality.  This yields (\ref{u-a}). Similar arguments also yield first (\ref{nabel-d-a}) and then (\ref{d-a}).

\medskip
\noindent{\it Step II}. We want to show the estimate of $\|\nabla u(t)\|_{L^6(\mathbb R^3_+)}$. To see this, first observe that by
(\ref{u-halft})$_1$, Lemma \ref{p-q},  and the interpolation inequality,
we have
\begin{eqnarray}
    &&\nonumber \label{nabel u-r}
    \|\nabla u(t)\|_{L^6(\mathbb R^3_+)}
    \\&\leq&\nonumber
     C_{1,6,2,3}(1+t)^{-1}\|u(\frac{t}2)\|_{L^2(\mathbb R^3_+)}
     +C\int_{\frac{t}{2}}^t(t-s)^{-\frac34}\|(u\cdot\nabla u-\nabla\cdot(\nabla d\odot\nabla d))(s)\|_{L^3(\mathbb R^3_+)}{\rm d}s
     \\&\leq&
     C_1t^{-\frac74}+\nonumber\\
     &&\ C\int_{\frac{t}{2}}^t(t-s)^{-\frac34}\big(\|u(s)\|_{L^6(\mathbb R^3_+)}\|\nabla u(s)\|_{L^6(\mathbb R^3_+)}
     +\|\nabla d(s)\|_{L^6(\mathbb R^3_+)}
     \|\nabla^2 d(s)\|_{L^6(\mathbb R^3_+)}\big){\rm d}s.
\end{eqnarray}
To bound the second term in the right hand of \eqref{nabel u-r},
we first estimate by using interpolation inequality, Lemma \ref{le:L3}, (\ref{nabel-u-assumption}) and (\ref{u-a}) that
\begin{eqnarray}
     \nonumber \label{K1}
     &&\int_{\frac{t}{2}}^t(t-s)^{-\frac34}\|u(s)\|_{L^6(\mathbb R^3_+)}
     \|\nabla u(s)\|_{L^6(\mathbb R^3_+)}{\rm d}s
     \\&\leq&\nonumber
     C\int_{\frac{t}{2}}^t(t-s)^{-\frac34}\|u(s)\|_{L^3(\mathbb R^3_+)}^{\frac13}
     \|u(s)\|_{L^{12}(\mathbb R^3_+)}^{\frac23}
     \|\nabla u(s)\|_{L^6(\mathbb R^3_+)}{\rm d}s
     \\&\leq&\nonumber
     C\varepsilon_0^{\frac{1}{3}}\int_{\frac{t}{2}}^t(t-s)^{-\frac34}s^{-\frac{11}{12}-\frac74}{\rm d}s
     \\&\leq&
     C\varepsilon_0^{\frac{1}{3}}t^{-\frac74}.
\end{eqnarray}
Similarly, we can estimate
\begin{eqnarray}
     \label{K2}
     \int_{\frac{t}{2}}^t(t-s)^{-\frac12}\|\nabla d(s)\|_{L^4(\mathbb R^3_+)}
     \|\nabla^2 d(s)\|_{L^4(\mathbb R^3_+)}{\rm d}s
     \leq
     C\varepsilon_0^{\frac1{3}}t^{-\frac74}.
\end{eqnarray}
Substituting the estimates \eqref{K1} and \eqref{K2}
into \eqref{nabel u-r} and choosing a sufficiently small $\varepsilon_0$,
we conclude that
\begin{eqnarray}
    \|\nabla u(t)\|_{L^6(\mathbb R^3_+)}\leq \frac32C_1t^{-\frac74}.
\end{eqnarray}

Finally, we want to estimate $\|\nabla^2 d(t)\|_{L^6(\mathbb R^3_+)}$.
By taking $\nabla$ of (\ref{system})$_3$, we have that
\begin{eqnarray}
    (\nabla d)_t-\Delta(\nabla d)=-\nabla(u\cdot\nabla d)+\nabla(|\nabla d|^2d)
\end{eqnarray}
so that by Duhamel's formula, we have
\begin{eqnarray}
    \nabla d(t)=e^{\frac{t}{2}\Delta}\nabla d(\frac{t}2)+\int_{\frac{t}{2}}^te^{(t-s)\Delta}\left(-\nabla(u\cdot\nabla d)+\nabla(|\nabla d|^2d)\right)(s){\rm d}s,
\end{eqnarray}
and
\begin{eqnarray}
    \label{nabel2-d}\nabla^2 d(t)=\nabla e^{\frac{t}{2}\Delta}\nabla d(\frac{t}2)+\int_{\frac{t}{2}}^t\nabla e^{(t-s)\Delta}\left(-\nabla(u\cdot\nabla d)+\nabla(|\nabla d|^2d)\right)(s){\rm d}s.
\end{eqnarray}
This, combined with Lemma \ref{p-q} and a similar argument as (\ref{K1}), yields that
\begin{eqnarray}
    \nonumber
    &&\|\nabla^2 d(t)\|_{L^6(\mathbb R^3_+)}
    \\&\leq&\nonumber
    C_{1,6,2,3}t^{-1}\|\nabla d(t)\|_{L^2(\mathbb R^3_+)}
    +\int_{\frac{t}{2}}^t(t-s)^{-\frac34}\big\|\nabla(u\cdot\nabla d)-\nabla(|\nabla d|^2d)\big\|_{L^{3}(\mathbb R^3_+)}{\rm d}s
    \\&\leq&\nonumber
    C_2t^{-\frac94}+\int_{\frac{t}{2}}^t(t-s)^{-\frac34}\big\||\nabla u||\nabla d|
    +|u||\nabla^2d|+|\nabla d||\nabla^2d|+|\nabla d|^3\big\|_{L^{3}(\mathbb R^3_+)}{\rm d}s
    \\&\leq&
    C_2t^{-\frac94}+K_1+K_2+K_3+K_4,
\end{eqnarray} where
\begin{eqnarray*}
    K_1&=&\int_{\frac{t}{2}}^t(t-s)^{-\frac34}\big\||\nabla u||\nabla d|\big\|_{L^{3}(\mathbb R^3_+)}{\rm d}s
    \\
    &\leq&\nonumber
    C\int_{\frac{t}{2}}^t(t-s)^{-\frac34}\|\nabla u(s)\|_{L^6(\mathbb R^3_+)}
    \|\nabla d(s)\|_{L^3(\mathbb R^3_+)}^{\frac13}
    \|\nabla d(s)\|_{L^{12}(\mathbb R^3_+)}^{\frac23}{\rm d}s
    \\&\leq&
    C\varepsilon_0^{\frac13}t^{-\frac94},
    \nonumber\\
    \\\nonumber
    K_2&=&\int_{\frac{t}{2}}^t(t-s)^{-\frac34}\big\||u||\nabla^2 d|\big\|
    _{L^{3}(\mathbb R^3_+)}{\rm d}s\\
    &\leq&\nonumber
    C\int_{\frac{t}{2}}^t(t-s)^{-\frac34}\|u(s)\|_{L^3(\mathbb R^3_+)}^{\frac13}
    \|u(s)\|_{L^{12}(\mathbb R^3_+)}^{\frac23}
    \|\nabla^2 d(s)\|_{L^6(\mathbb R^3_+)}{\rm d}s
    \\&\leq&
    C\varepsilon_0^{\frac13}t^{-\frac94},\nonumber\\
    \\ \nonumber
    K_3&=&\int_{\frac{t}{2}}^t(t-s)^{-\frac34}\big\||\nabla d||\nabla^2 d|\big\|_{L^{3}(\mathbb R^3_+)}{\rm d}s
    \\
    &\leq&\nonumber
    C\int_{\frac{t}{2}}^t(t-s)^{-\frac34}\|\nabla d(s)\|_{L^3(\mathbb R^3_+)}^{\frac13}\|\nabla d(s)\|_{L^{12}(\mathbb R^3_+)}^{\frac23}
    \|\nabla^2 d(s)\|_{L^6(\mathbb R^3_+)}{\rm d}s
    \\&\leq&
    C\varepsilon_0^{\frac13}t^{-\frac94},
    \nonumber\\
    \\ \nonumber
    K_4&=&\int_{\frac{t}{2}}^t(t-s)^{-\frac34}
    \big\|\nabla d(s)\big\|_{L^9(\mathbb R^3_+)}^3{\rm d}s
    \\&\leq&\nonumber
    \int_{\frac{t}{2}}^t(t-s)^{-\frac34}\|\nabla d(s)\|_{L^3(\mathbb R^3_+)}^{\frac13}\|\nabla d(s)\|_{L^{12}(\mathbb R^3_+)}^{\frac83}{\rm d}s
    \\&\leq&
    C\varepsilon_0^{\frac13}t^{-\frac94}.
\end{eqnarray*}
Thus (\ref{c5}) follows by choosing $\varepsilon_0$ sufficiently small.
This completes the proof. \end{proof}

Now we can apply the continuity argument to show the following Corollary.

\begin{corollary} \label{cor-4}
 Under the same assumptions of Theorem 1.1, if
$(u,d)$ is the global strong solution obtained by Theorem 1.1, then
the following estimates
\begin{eqnarray}
   &&\|u(t)\|_{L^{r}(\mathbb R^3_+)}
   \leq Ct^{-\frac32(1-\frac{1}{r})},\label{u-a-2}, \\
   &&\label{d-a-2}\|(d-e_3)(t)\|_{L^r(\mathbb R^3_+)}
   \leq Ct^{-\frac32(1-\frac{1}{r})},\\
   &&\|\nabla d(t)\|_{L^r(\mathbb R^3_+)}
   \leq Ct^{-\frac12-\frac32(1-\frac{1}{r})},\label{nabel-d-a-2}\\
   &&\label{c4-2}\|\nabla u(t)\|_{L^6(\mathbb R^3_+)}\leq Ct^{-\frac74},\\
   &&\label{c5-2}\|\nabla^2 d(t)\|_{L^6(\mathbb R^3_+)}\leq Ct^{-\frac94},
\end{eqnarray}
hold for any $t>0$, where $r\in [2,\infty]$.
\end{corollary}

Based on the estimate \eqref{u-a-2}, we can deduce

\begin{corollary} Under the same assumptions as Corollary 3.4,
it holds that for any $t>0$ and $r\in (1,6]$,
\begin{eqnarray}
    &&\nonumber \label{nabel u-r-2}
    \|\nabla u(t)\|_{L^r(\mathbb R^3_+)}
    \le Ct^{-\frac12-\frac32(1-\frac{1}{r})}.
\end{eqnarray}
\end{corollary}
\begin{proof} In fact, it follows from (\ref{u-halft})$_1$ and Lemma 2.1 that, for any $r\in (1,6)$,
\begin{eqnarray*}
    &&\|\nabla u(t)\|_{L^r(\mathbb R^3_+)}\\
    &&\leq\nonumber
     C(1+t)^{-\frac12}\|u(\frac{t}2)\|_{L^r(\mathbb R^3_+)}
     +C\int_{\frac{t}{2}}^t(t-s)^{-\frac12}\big\||u||\nabla u|+|\nabla d||\nabla^2 d|
     \big\|_{L^{r}(\mathbb R^3_+)}{\rm d}s
     \\&&\leq
     Ct^{-\frac12-\frac32(1-\frac{1}{r})}\\
     &&+C\int_{\frac{t}{2}}^t(t-s)^{-\frac12}\big(\|u(s)\|_{L^{\frac{6r}{6-r}}(\mathbb R^3_+)}\|\nabla u(s)\|_{L^6(\mathbb R^3_+)}+\|\nabla d(s)\|_{L^\frac{6r}{6-r}(\mathbb R^3_+)}\|\nabla^2 d(s)\|_{L^6(\mathbb R^3_+)}\big){\rm d}s
     \\
     &&\leq
     Ct^{-\frac12-\frac32(1-\frac{1}{r})}.
\end{eqnarray*}
This, combined with (\ref{c4-2}), completes the proof. \end{proof}

We also enlarge the range for the estimate of $\nabla^2 d(t)$. More precisely,
\begin{corollary} Under the same assumptions as Corollary 3.4,
it holds that for any $t>0$ and $s\in [1,6]$,
\begin{eqnarray}
    \|\nabla^2d(t)\|_{L^s(\mathbb R^3_+)}
    \leq Ct^{-1-\frac32(1-\frac{1}{s})}.
\end{eqnarray}
\end{corollary}
\begin{proof} From (\ref{nabel2-d}), Lemma \ref{le:1}, (\ref{nabel d-L1}), and Corollary \ref{cor-4}, we obtain that
\begin{eqnarray*}
    &&\|\nabla^2 d(t)\|_{L^1(\mathbb R^3_+)}
    \\&\leq&\nonumber
    C_{1,1,1,3}t^{-\frac12}\|\nabla d(\frac{t}2)\|_{L^1(\mathbb R^3_+)}\\
    &+&\int_{\frac{t}{2}}^t(t-s)^{-\frac12}\big\||\nabla(u\cdot\nabla d)|+|\nabla(|\nabla d|^2d)|\big\|_{L^1(\mathbb R^3_+)}{\rm d}s
    \\
    &\leq&\nonumber
    Ct^{-1}+\int_{\frac{t}{2}}^t(t-s)^{-\frac12}\big(\|\nabla u(s)\|_{L^2(\mathbb R^3_+)}\|\nabla d(s)\|_{L^2(\mathbb R^3_+)}
    +\|u(s)\|_{L^2(\mathbb R^3_+)}\|\nabla^2 d(s)\|_{L^2(\mathbb R^3_+)}
    \\
    &&\nonumber
    \qquad\qquad+\|\nabla d(s)\|_{L^2(\mathbb R^3_+)}
    \|\nabla^2 d(s)\|_{L^2(\mathbb R^3_+)}
    +\|\nabla d(s)\|_{L^3(\mathbb R^3_+)}^3\big)(s){\rm d}s
    \\&\leq&
    Ct^{-1}.
\end{eqnarray*}
The conclusion now follows from (\ref{c5-2}) and the interpolation inequality.
\end{proof}

Combining all conclusions in this section, we prove the time decay estimates in Theorem \ref{th:E_0}. 

\bigskip

\section{Proof of Theorem 1.2}

In this section, we will give a proof of Theorem \ref{th:E_1}, which follows from Lemma 4.1 below.

Under the assumptions in Theorem 1.2, by integrating (\ref{nabelu+dt}) over $[0,t]$, 
and applying (\ref{le:B}), (\ref{d-t}), (\ref{nabla.2d}) and the elliptic estimates, we have
that for any $t>0$,
\begin{eqnarray}
   &&\nonumber
   \label{4.1}
   \big(\|\nabla u(t)\|_{L^2(\mathbb R^3_+)}^2+\|\nabla^2 d(t)\|_{L^2(\mathbb R^3_+)}^2+\|d_t(t)\|_{L^2(\mathbb R^3_+)}^2\big)
   \\&&
   +\int_0^t\big(\|u_t\|_{L^2(\mathbb R^3_+)}^2+\|\nabla u\|_{H^1(\mathbb R^3_+)}^2+\|d_t\|_{H^1(\mathbb R^3_+)}^2+\|\nabla^2 d\|_{H^1(\mathbb R^3_+)}^2\big){\rm d}s\leq C.
\end{eqnarray}

\begin{lemma} Under the same assumptions as in Theorem 1.2, it holds that for any $t>0$,
\begin{eqnarray}
     \|\nabla u(t)\|_{L^1(\mathbb R^3_+)}\leq Ct^{-\frac12}.
\end{eqnarray}
\end{lemma}
\begin{proof} Without loss of generality, we may assume that $t\geq1$. By (\ref{u})$_1$, (\ref{decomposition})--(\ref{P}), the H\"older inequality, (\ref{4.1}) and Lemma \ref{le:higher order}, we obtain
\begin{eqnarray}
    \nonumber
    &&\|\nabla u\|_{L^1(\mathbb R^3_+)}
    \\&\leq&\nonumber
     Ct^{-\frac12}\|u_0\|_{L^1(\mathbb R^3_+)}+\int_0^t\|\nabla e^{-(t-s)\mathbb{A}}\mathbb{P}\left(u\cdot\nabla u+\nabla\cdot(\nabla d\odot\nabla d)\right)(s)\|_{L^1(\mathbb R^3_+)}{\rm d}s
    \\&\leq&\nonumber
    Ct^{-\frac12}\|u_0\|_{L^1(\mathbb R^3_+)}
    \\&& \nonumber
    +C\int_0^t(t-s)^{-\frac12}\big(\|u\|_{H^1(\mathbb R^3_+)}^2+\|\nabla d\|_{H^1(\mathbb R^3_+)}^2\nonumber\\
    &&\qquad\qquad\qquad\qquad+\||\nabla d||\nabla^2 d|\|_{L^1(\mathbb R^3_+)}+\||\nabla d||\nabla^3d|\|_{L^1(\mathbb R^3_+)}\big)(s){\rm d}s\nonumber
    \\&\leq&\nonumber
    Ct^{-\frac12}+C\int_0^t(t-s)^{-\frac12}\big(\|u(s)\|_{H^1(\mathbb R^3_+)}^2+\|\nabla d(s)\|_{H^1(\mathbb R^3_+)}^2\nonumber\\
    &&\qquad\qquad\qquad+\|\nabla d(s)\|_{L^2(\mathbb R^3_+)}\|\nabla^2 d(s)\|_{L^2(\mathbb R^3_+)}
    +\|\nabla d(s)\|_{L^2(\mathbb R^3_+)}\|\nabla^3d(s)\|_{L^{2}(\mathbb R^3_+)}\big){\rm d}s\nonumber
    \\&\leq&\nonumber
    Ct^{-\frac12}+C\int_0^t(t-s)^{-\frac12}(1+s)^{-\frac54}{\rm d}s\nonumber\\
    &&\qquad\qquad\qquad+C\big(\int_0^{\frac{t}{2}}
    +\int_{\frac{t}{2}}^{t}\big)(t-s)^{-\frac12}\|\nabla d(s)\|_{L^2(\mathbb R^3_+)}\|\nabla^3d(s)\|_{L^{2}(\mathbb R^3_+)}{\rm d}s
    \nonumber\\
    &\leq&\nonumber
    Ct^{-\frac12}+C\big(\frac{t}{2}\big)^{-\frac12}\big(\int_0^{\frac{t}{2}}\|\nabla d(s)\|_{L^2(\mathbb R^3_+)}^2{\rm d}s\big)^{\frac12}\big(\int_0^{\frac{t}{2}}\|\nabla^3 d(s)\|_{L^2(\mathbb R^3_+)}^2{\rm d}s\big)^{\frac12}
    \\&&\nonumber
    +C\sup\limits_{s\in\left[\frac{t}{2},t\right]}\|\nabla^3 d(s)\|_{L^2(\mathbb R^3_+)}^2\int_{\frac{t}{2}}^{t}(t-s)^{-\frac12}\|\nabla d(s)\|_{L^2(\mathbb R^3_+)}{\rm d}s
    \\&\leq&\nonumber
    Ct^{-\frac12}+Ct^{-\frac12}\big(\int_0^{\frac{t}{2}}(1+s)^{-\frac54}{\rm d}s\big)^{\frac12}
    \\&\leq&
    Ct^{-\frac12}.
\end{eqnarray}
This completes the proof of this Lemma.
\end{proof}

\section{Proof of Corollary 1.1}

In this section, we will prove Corollary 1.1.
In fact, the conclusions of Corollary \ref{th:E_2} follow from Lemma 5.1 and Lemma 5.2 below. 

We first recall an revised estimates with weighted condition (\ref{weight}). For any $p\in (1,+\infty)$, it holds
\begin{eqnarray} \label{new-estimate}
     &&\|\nabla^ke^{-t\mathbb{A}}u_0\|_{L^p(\mathbb{R}^3_+)}
     \leq
     Ct^{-\frac{k+1}{2}-\frac32\left(1-\frac{1}{p}\right)}\int_{\mathbb R^3_+}(1+x_3)|u_0(x)|{\rm d}x,
\end{eqnarray} for $k=0,1$.

\begin{lemma} \label{cor-new-1} Under the same assumptions of Theorem \ref{th:E_2}, if
$(u,d)$ is the global strong solution of \eqref{system} obtained by Theorem \ref{th:E_0}, then, for any $r\in (1,3/2)$, it holds
that
\begin{eqnarray}
     \|u(t)\|_{L^{r}(\mathbb R^3_+)}
     \leq Ct^{-\frac12-\frac32\left(1-\frac{1}{r}\right)}.
\end{eqnarray}
\end{lemma}
\begin{proof} By using (\ref{u-halft}), (\ref{3.49}), (\ref{new-estimate}) and Theorem \ref{th:E_0}, we have,
\begin{eqnarray}
    \nonumber
    \|u(t)\|_{L^{r}(\mathbb R^3_+)}&\leq& Ct^{-\frac12-\frac32\left(1-\frac{1}{r}\right)}\int_{\mathbb R^3_+}(1+x_3)|u_0(x)|{\rm d}x
    \\&&\nonumber
    +\int_0^t(t-s)^{-\frac12-\frac32(1-\frac{1}{r})}
     \big(\|u(s)\|_{L^2(\mathbb R^3_+)}^2+\|\nabla d(s)\|_{L^2(\mathbb R^3_+)}^2\big){\rm d}s
     \\&\leq&\nonumber
     Ct^{-\frac12-\frac32\left(1-\frac{1}{r}\right)}+\int_0^t(t-s)^{-\frac12-\frac32(1-\frac{1}{r})}
     (1+s)^{-\frac32}{\rm d}s
     \\&\leq&
     Ct^{-\frac12-\frac32\left(1-\frac{1}{r}\right)}.
\end{eqnarray} This completes the proof.
\end{proof}

\begin{lemma} \label{cor-new-2} Under the same assumptions of Theorem \ref{th:E_2}, if
$(u,d)$ is the global strong solution of \eqref{system}
obtained by Theorem \ref{th:E_0}, then, for any $r\in [\frac32,\infty]$ and $p\in (1,6]$, it holds
that
\begin{eqnarray}
     \label{improve-u}
     &&\|u(t)\|_{L^{r}(\mathbb R^3_+)}
     \leq Ct^{-\frac12-\frac32\left(1-\frac{1}{r}\right)},\\
     \label{improve-u-2}
     &&\|\nabla u(t)\|_{L^{p}(\mathbb R^3_+)}
     \leq Ct^{-1-\frac32\left(1-\frac{1}{p}\right)}.
\end{eqnarray}
\end{lemma}
\begin{proof}
Without loss of generality, assume that $t\geq 1$. On one hand, for $k=0,1$, from Lemma \ref{le:1} and Lemma \ref{cor-new-1}, it holds, for any $r\in [\frac32,\infty]$,
\begin{eqnarray}
     \nonumber \label{improve-linear-u}
     \|\nabla^ke^{-\frac{t}{2}\mathbb{A}}u(\frac{t}2)\|_{L^r(\mathbb{R}^3_+)}
     &\leq&\nonumber
     Ct^{-\frac{k}{2}-\frac32\left(\frac{1}{\tilde{r}}-\frac{1}{r}\right)}\|u(\frac{t}2)\|_{L^{\tilde{r}}(\mathbb R^3_+)}
     \\&\leq&\nonumber
      Ct^{-\frac{k}{2}-\frac32\left(\frac{1}{\tilde{r}}-\frac{1}{r}\right)}t^{-\frac12-\frac32\left(1-\frac{1}{\tilde{r}}\right)}
      \\&\leq&
      Ct^{-\frac{k+1}{2}-\frac32\left(1-\frac{1}{r}\right)}.
\end{eqnarray} Here we choose some $\tilde{r}\in (1,\frac32)$. 

On the other hand, by using Lemma \ref{le:1}, Lemma \ref{key-estimate} and Theorem \ref{th:E_0}, 
we have, for any $r\in [\frac32,\infty)$, it holds that,
\begin{eqnarray}
     \label{improve-nonlinear-1}
     \nonumber&&\big\|\int_{\frac{t}{2}}^te^{-(t-s)\mathbb{A}}\mathbb{P}\big(u\cdot\nabla u-\nabla\cdot(\nabla d\odot\nabla d)\big)(s){\rm d}s\big\|_{L^r(\mathbb R^3_+)}
     \\&\leq&\nonumber
     \int_{\frac{t}{2}}^t(t-s)^{-\frac32(\frac{2}{3}-\frac{1}{r})}\big(\|u(s)\cdot\nabla u(s)\|_{L^{\frac32}(\mathbb R^3_+)}+\|\nabla\cdot(\nabla d\odot\nabla d)(s)\|_{L^{\frac32}(\mathbb R^3_+)}\big){\rm d}s
     \\&\leq&\nonumber
     \int_{\frac{t}{2}}^t(t-s)^{-\frac32(\frac{2}{3}-\frac{1}{r})}\big(\|u(s)\|_{L^{2}(\mathbb R^3_+)}\|\nabla u(s)\|_{L^6(\mathbb R^3_+)}+\|\nabla d(s)\|_{L^{2}(\mathbb R^3_+)}\|\nabla^2 d(s)\|_{L^6(\mathbb R^3_+)}\big){\rm d}s
     \\&\leq&
     C\int_{\frac{t}{2}}^t(t-s)^{-\frac32\left(\frac{2}{3}-\frac{1}{r}\right)}s^{-\frac34-\frac74}{\rm d}s
     \leq
     Ct^{-\frac12-\frac32\left(1-\frac{1}{r}\right)},
\end{eqnarray} and
\begin{eqnarray}
     \label{improve-nonlinear-2}
     \nonumber&&\big\|\int_{\frac{t}{2}}^te^{-(t-s)\mathbb{A}}\mathbb{P}\left(u\cdot\nabla u-\nabla\cdot(\nabla d\odot\nabla d)\right)(s){\rm d}s\big\|_{L^\infty(\mathbb R^3_+)}
     \\&\leq&\nonumber
     \int_{\frac{t}{2}}^t(t-s)^{-\frac12}\big(\|u(s)\cdot\nabla u(s)\|_{L^{3}(\mathbb R^3_+)}
     +\|\nabla\cdot(\nabla d\odot\nabla d)(s)\|_{L^{3}(\mathbb R^3_+)}\big){\rm d}s
     \\&\leq&\nonumber
     \int_{\frac{t}{2}}^t(t-s)^{-\frac12}\big(\|u(s)\|_{L^{6}(\mathbb R^3_+)}\|\nabla u(s)\|_{L^6(\mathbb R^3_+)}
     +\|\nabla d(s)\|_{L^{6}(\mathbb R^3_+)}\|\nabla^2 d(s)\|_{L^6(\mathbb R^3_+)}\big){\rm d}s
     \\&\leq&
     \int_{\frac{t}{2}}^t(t-s)^{-\frac12}s^{-\frac54-\frac74}{\rm d}s
     \leq
     Ct^{-2},
\end{eqnarray}
Then (\ref{improve-u}) follows from (\ref{u-halft})$_1$, (\ref{improve-linear-u}), (\ref{improve-nonlinear-1}) and (\ref{improve-nonlinear-2}).

Furthermore, for any $p\in(1,6]$ and for some $\tilde{p}$ such that $\frac{3p}{3+p}<\tilde{p}<p$, one obtains
\begin{eqnarray}
     \label{improve-nonlinear-3}
     \nonumber&&\big\|\int_{\frac{t}{2}}^t\nabla e^{-(t-s)\mathbb{A}}\mathbb{P}\left(u\cdot\nabla u-\nabla\cdot(\nabla d\odot\nabla d)\right)(s){\rm d}s\big\|_{L^p(\mathbb R^3_+)}
     \\&\leq&\nonumber
     \int_{\frac{t}{2}}^t(t-s)^{-\frac12-\frac32(\frac{1}{\tilde{p}}-\frac{1}{p})}
     \big(\|u(s)\cdot\nabla u(s)\|_{L^{\tilde{p}}(\mathbb R^3_+)}+\|\nabla\cdot(\nabla d\odot\nabla d)(s)\|_{L^{\tilde{p}}(\mathbb R^3_+)}\big){\rm d}s
     \\&\leq&\nonumber
     \int_{\frac{t}{2}}^t(t-s)^{-\frac12-\frac32(\frac{1}{\tilde{p}}-\frac{1}{p})}
     \big(\|u(s)\|_{L^{\frac{6\tilde{p}}{6-\tilde{p}}}(\mathbb R^3_+)}\|\nabla u(s)\|_{L^6(\mathbb R^3_+)}\nonumber\\
     &&\qquad+\|\nabla d(s)\|_{L^{\frac{6\tilde{p}}{6-\tilde{p}}}(\mathbb R^3_+)}\|\nabla^2 d(s)\|_{L^6(\mathbb R^3_+)}\big){\rm d}s\nonumber
     \\&\leq&
     \int_{\frac{t}{2}}^t(t-s)^{-\frac12-\frac32(\frac{1}{\tilde{p}}-\frac{1}{p})}s^{-\frac12-\frac32\left(1-\frac{6-\tilde{p}}{6\tilde{p}}\right)}s^{-\frac74}{\rm d}s
     \leq
     Ct^{-1-\frac32\left(1-\frac{1}{p}\right)}.
\end{eqnarray}
Then (\ref{improve-u-2}) follows from (\ref{improve-linear-u}) with $k=1$ and (\ref{improve-nonlinear-3}). \end{proof}

\noindent{\bf Acknowledgements}. Huang and Wen would like to thank Professor Pigong Han for helpful discussions.
Huang is partially supported by NSF of China (No.11401439), the Foundation for
Distinguished Young Talents in Higher Education of Guangdong,
China (No. 2014KQNCX162). Wen is partially supported by the
NSF of China (Grant No. 11722104,
11671150), and by GDUPS (2016) and the Fundamental Research Funds
for the Central Universities of China (Grant No. D2172260). Wang is partially supported by
NSF grant 1764417. Parts of this work was conducted when Huang visited South China University
of Technology in Fall 2017, and when Huang and Wen visited
Purdue University in Spring 2018. Huang and Wen would like to thank the departments for their hospitality.

%
%\section*{References}

\end{document}